\documentclass[10pt]{amsart}

\usepackage{latexsym}
\usepackage{amsmath}
\usepackage{amsfonts}
\usepackage{amssymb}
\usepackage{mathrsfs}
\usepackage{bm}
\usepackage{comment} 
\usepackage{accents}

\usepackage{tikz}
\usetikzlibrary{matrix,arrows}
\usepackage{scalefnt}

\newcommand{\tw}[1] {(\!(#1)\!)}

\newtheorem{theorem}{Theorem}[section]
\newtheorem{lemma}[theorem]{Lemma}
\newtheorem{corollary}[theorem]{Corollary}
\newtheorem{proposition}[theorem]{Proposition}
\newtheorem{noname*}[theorem]{}

\theoremstyle{definition}
\newtheorem{definition}[theorem]{Definition}

\theoremstyle{remark}
\newtheorem{remark}[theorem]{Remark}

\numberwithin{equation}{section}

\DeclareMathOperator{\Hom}{Hom}
\DeclareMathOperator{\dgHom}{\mathcal{H}om}
\DeclareMathOperator{\dgEnd}{\mathcal{E}nd}

\DeclareMathOperator{\Ext}{Ext}
\DeclareMathOperator{\uHom}{\underline{Hom}}
\DeclareMathOperator{\DN}{D^{\textup{b}}_{\textup{m}}(\mathcal{N}_0)}

\DeclareMathOperator{\DNG}{D^{\textup{b}}_{\mathit{G}, \textup{m}}(\mathcal{N}_0)}
\DeclareMathOperator{\DBG}{D^{\textup{b}}_{\mathit{G}, \textup{m}}(\mathscr{B}_0)}
\DeclareMathOperator{\DX}{D^{\textup{b}}_{\textup{m}}(\mathit{X}_0)}

\DeclareMathOperator{\CX}{C^b(\Pure_{\mathscr{S}}\mathit{X}_0)}
\DeclareMathOperator{\KN}{K^b(\Pure\mathcal{N}_0)}
\DeclareMathOperator{\KNG}{K^b(\Pure_\mathit{G}\mathcal{N}_0)}
\DeclareMathOperator{\KNGG}{K^b(\Pure_{\mathit{G}\times\mathbb{G}_m}\mathcal{N}_0)}
\DeclareMathOperator{\KBG}{K^b(\Pure_G\mathscr{B}_0)}
\DeclareMathOperator{\KX}{K^b(\mathcal{P}ure_{\mathscr{S}}\mathit{X}_0)}

\DeclareMathOperator{\PMX}{\mathcal{P}erv^{mix}(X_0)}
\DeclareMathOperator{\Ql}{\bar{\mathbb{Q}}_\ell}
\DeclareMathOperator{\uH}{\underline{H}}
\DeclareMathOperator{\cH}{H}
\DeclareMathOperator{\Pure}{\mathcal{P}ure}
\DeclareMathOperator{\gr}{gr}
\DeclareMathOperator{\IC}{IC}
\DeclareMathOperator{\DNbar}{D^{b}_{Spr}(\mathcal{N})}
\DeclareMathOperator{\DNc}{D^{b}_c(\mathcal{N})}
\DeclareMathOperator{\DBGc}{D^{b}_{\mathit{G},c}(\mathscr{B})}
\DeclareMathOperator{\DBc}{D^{b}_{c}(\mathscr{B})}
\DeclareMathOperator{\DNGbar}{D^{b}_{\mathit{G}, Spr}(\mathcal{N})}
\DeclareMathOperator{\DNGc}{D^{b}_{\mathit{G}, c}(\mathcal{N})}
\DeclareMathOperator{\DXbar}{D^{b}_{\mathscr{S}}(\mathit{X})}
\DeclareMathOperator{\Pervs}{\mathcal{P}erv_{\mathscr{S}}}
\DeclareMathOperator{\Perv}{\mathcal{P}erv}
\DeclareMathOperator{\spr}{\mathcal{S}pr}
\DeclareMathOperator{\DXc}{D^{b}_c(\mathit{X})}
\DeclareMathOperator{\DXGc}{D^{b}_{G, c}(\mathit{X})}
\DeclareMathOperator{\id}{id}
\DeclareMathOperator{\Aut}{Aut}
\DeclareMathOperator{\End}{End}
\DeclareMathOperator{\R}{\Ql [\mathit{W}]\#H^\bullet(\mathscr{B})}
\DeclareMathOperator{\Coinv}{Coinv(\mathit{W})}
\DeclareMathOperator{\Sh}{S\mathfrak{h}^*}
\DeclareMathOperator{\Rtwo}{\Ql [\mathit{W}]\# \Coinv}
\DeclareMathOperator{\RGtwo}{\Ql [\mathit{W}]\#\Sh}
\DeclareMathOperator{\RG}{\Ql[\mathit{W}]\#\coGflag}
\DeclareMathOperator{\RK}{\Ql [\mathit{W}]\#\bigwedge\mathfrak{h}}

\DeclareMathOperator{\QlW}{\Ql [\mathit{W}]}
\DeclareMathOperator{\coflag}{H^\bullet(\mathscr{B})}
\DeclareMathOperator{\coGflag}{H^\bullet_\mathit{G}(\mathscr{B})}
\DeclareMathOperator{\purex}{\Pure_{\mathscr{S}}(\mathit{X}_0)}
\DeclareMathOperator{\puren}{\Pure\mathcal{N}_0}
\DeclareMathOperator{\puregn}{\Pure_\mathit{G}\mathcal{N}_0}

\DeclareMathOperator{\hotproj}{K^{b}(gProj \Rtwo)}
\DeclareMathOperator{\hotprojG}{K^{b}(gProj \RGtwo)}
\DeclareMathOperator{\gmod}{gMod}
\DeclareMathOperator{\gproj}{gProj}
\DeclareMathOperator{\dSh}{D^b(\gmod \Sh)}
\DeclareMathOperator{\dA}{D^b(\mathcal{A})}
\DeclareMathOperator{\dAG}{D^b(\mathcal{A}_\mathit{G})}
\DeclareMathOperator{\dAGG}{D^b(\mathbb{H})}
\DeclareMathOperator{\dperA}{D^b_{per}(\mathcal{A})}
\DeclareMathOperator{\dperAG}{D^b_{per}(\mathcal{A}_\mathit{G})}

\DeclareMathOperator{\dperRwhat}{D^b_{per}(\gmod \R)}
\DeclareMathOperator{\dRG}{D^b(\gmod \RGtwo)}
\DeclareMathOperator{\dRGwhat}{D^b(\gmod \RG)}

\DeclareMathOperator{\dRK}{D^b(\gmod \RK)}
\DeclareMathOperator{\pervk}{\mathcal{P}erv_{KD}}
\DeclareMathOperator{\pervkn}{\mathcal{P}erv_{KD}(\mathcal{N}_0)}
\DeclareMathOperator{\dpervk}{D^b(\pervkn)}

\DeclareMathOperator{\dgder}{\mathcal{D}^{dg}}
\DeclareMathOperator{\dgK}{\mathcal{K}P}
\DeclareMathOperator{\dgf}{\mathcal{D}^{dg}_{f}}
\DeclareMathOperator{\dgper}{\mathcal{D}^{dg}_{per}}
\DeclareMathOperator{\der}{D_{c}^{b}}
\DeclareMathOperator{\constant}{\mathcal{C}}
\DeclareMathOperator{\orlov}{K^b(\mathscr{A})}
\begin{document}

\title[Formality and a derived Springer correspondence]{Formality for the nilpotent cone\\ and a derived Springer correspondence}

\author{Laura Rider}
\address{Mathematics Department\\
Louisiana State University\\
Baton Rouge, Louisiana}
\email{lrider1@math.lsu.edu}


\begin{abstract}

Recall that the Springer correspondence relates representations of the Weyl group to perverse sheaves on the nilpotent cone.  We explain how to extend this to an equivalence between the triangulated category generated by the Springer perverse sheaf and the derived category of differential graded modules over a dg-ring related to the Weyl group. 
\end{abstract}

\maketitle

\section{Introduction}
\label{introduction}

An important problem in geometric representation theory is describing the (equivariant) derived category of sheaves on a variety attached to an algebraic group. For instance, this has been done by Bernstein and Lunts in \cite{BL} for $pt$, by Lunts in \cite{Lu} for projective and affine toric varieties, by Arkhipov, Bezrukavnikov, and Ginzburg in \cite{ABG} for the affine Grassmannian, and by Schn\"urer in \cite{S2} for flag varieties. 

We consider this problem for the nilpotent cone $\mathcal{N}$ of a connected reductive algebraic group $G$. In particular, we focus on $\DNGbar$ --- the triangulated subcategory of $\DNGc$ generated by the simple summands of the Springer perverse sheaf $\mathbf{A}$. It is in this setting that we prove Theorem \ref{bigtheorem}: there is an equivalence of triangulated categories 
\begin{align}\label{thm}
	\DNGbar \cong \dgf(\RG).
\end{align} Here $\dgf(\RG)$ is the derived category of finitely generated differential graded (dg) modules over the smash product algebra $\RG$ with $W$, the Weyl group of $G$ and $\coGflag$, the $G$-equivariant cohomology of the flag variety. This theorem can be viewed as a derived version of the Springer correspondence. Along the way, we prove the following ``mixed version'' of the above (see Theorem \ref{mixedSpringer}): an equivalence of triangulated categories \[\KNG \cong \dRGwhat\] relating a category built out of mixed sheaves $\KNG$ and the derived category of graded modules over $\RG$. We also prove the obvious non-equivariant analogue, i.e. an equivalence \[\KN\cong\dperRwhat.\]  

There are two key components to the proof of \eqref{thm}: formality and Koszulity. A dg-ring $\mathcal{R}$ is called formal if it is quasi-isomorphic to its cohomology $\cH^\bullet(\mathcal{R})$. The role of formality in derived equivalences such as \eqref{thm} is well established, but in our case, the construction of the dg-ring and functor is less straightforward than the analogue for the flag variety. Roughly, the machinery of quasi-hereditary categories does not apply, so new techniques are required. Instead, we exploit a non-standard $t$-structure on the triangulated category $\KNG$ that arises via Koszul duality. 

One might also expect a non-equivariant version of \eqref{thm}. Unfortunately, the ring $\R$ is not Koszul, so the methods of the present paper do not yield that result. 

\subsection*{Organization of the paper} Let $X_0$ be a variety over $\mathbb{F}_q$. We introduce a category $\Pure_{\mathscr{S}}(X_0)\subset\DX$ of weight zero objects in Section \ref{basics} and construct a realization functor $\KX\rightarrow\DX$ in Section \ref{section:realization}. In Section \ref{section:Mix}, we prove that $\KX$ is a \textit{mixed version} of its analogue over $\bar{\mathbb{F}}_q$. In Section \ref{section:Springer}, we introduce notation related to $\mathcal{N}$ and prove Frobenius invariance of certain $\Ext$ groups for objects related to the Springer sheaf $\mathbf{A}$ (see Lemmas \ref{lemma:trivialFrob} and \ref{Lemma:purenfrinv}). We prove a ``mixed version" of the Springer correspondence (see Theorem \ref{mixedSpringer}) in Section \ref{sec:mixedspringer}, and in Section \ref{dg}, we prove a ``derived version" of the Springer correspondence (Theorem \ref{bigtheorem}). In the appendix, we prove our functor for the equivalence in Theorem \ref{bigtheorem} is triangulated.

\subsection*{Acknowledgments}
This article contains the main results of my thesis written at Louisiana State University. I am very grateful to my advisor Pramod Achar for his advice concerning this problem and for his adamant insistence on correct dash usage. Without his guidance, this work would not be possible. I would also like to thank Simon Riche and Anthony Henderson for a number of useful comments and suggestions.

\section{Basics on mixed sheaves and purity}
\label{basics}

We fix a finite field $\mathbb{F}_q$ and a prime number $\ell$ different from the characteristic of $\mathbb{F}_q$. Let $X_0$ be a variety defined over $\mathbb{F}_q$. We study the (bounded) $\ell$-adic derived category consisting of mixed complexes of sheaves constructible with respect to some fixed stratification of $X_0$. This category, denoted $\DX$, is studied extensively in \cite{BBD,D, BGS, AR} and others. Note that in Section \ref{section:Mix}, we will use a different meaning for the word \textit{mixed} (\cite[Section 4]{BGS}). While $\DX$ shares some characteristics of that definition (i.e. the notion of weights and purity for objects), it is not mixed in the sense of \cite{BGS}.

If we have a sheaf, complex of sheaves, or a perverse sheaf $\mathscr{F}_0\in\DX$, it will often be useful to extend scalars to get an object $\mathscr{F}$ in $\der(X)$, where $X:= X_0\times_{\textup{Spec}\mathbb{F}_q}\textup{Spec}(\bar{\mathbb{F}}_q)$. Let $Fr: X\rightarrow X$ be the Frobenius map. Note that after extending scalars, $\mathscr{F}$ is endowed with additional structure: an isomorphism $Fr^*(\mathscr{F})\cong \mathscr{F}$. Let $a: X_0\rightarrow \textup{Spec}(\mathbb{F}_q)$ be the structure map. For $\mathscr{F}_0$ and $\mathscr{G}_0$ in $\DX$, we let $\uHom^i(\mathscr{F}_0,\mathscr{G}_0) = R^ia_*R\mathcal{H}om(\mathscr{F}_0, \mathscr{G}_0)$. This is a vector space with an action of Frobenius. Forgetting that action yields $\Hom_{\der(X)}(\mathscr{F}, \mathscr{G}[i])$. 

Recall, from \cite[5.1.2.5]{BBD}, that for $\mathscr{F}_0$ and $\mathscr{G}_0$ in $\DX$, we have a short exact sequence relating morphisms in $\DX$ to the Frobenius coinvariants and invariants (i.e. the cokernel and kernel of the map $Fr-Id$) of the morphisms  in $\der(X)$.%%%%% equation Frob SES %%%%%%%%%%%%%% 
\begin{equation}
\label{eq:FrobSES}
0\rightarrow \uHom^{i-1}(\mathscr{F},\mathscr{G})_{\textup{Frob}}\rightarrow \Hom^i(\mathscr{F}_0,\mathscr{G}_0)\rightarrow \uHom^i(\mathscr{F},\mathscr{G})^{\textup{Frob}}\rightarrow 0
\end{equation}

\begin{remark}
\label{rm:purity} The Frobenius invariants $\uHom^i(\mathscr{F}, \mathscr{G})^{\textup{Frob}}$ inject into the weight 0 part of $\uHom^i(\mathscr{F}, \mathscr{G})$ and the Frobenius coinvariants $\uHom^i(\mathscr{F},\mathscr{G})_{\textup{Frob}}$ are a quotient of the weight 0 part. Thus, $\uHom^i(\mathscr{F}, \mathscr{G})^{\textup{Frob}}$ and $\uHom^i(\mathscr{F}, \mathscr{G})_{\textup{Frob}}$ vanish when $\uHom^i(\mathscr{F}, \mathscr{G})$ is pure of non-zero weight. 
\end{remark}

Fix a square-root of the Tate sheaf. Note that Tate twist affects the weights of an object in the following way: for $\mathscr{F}_0\in\DX$, the weights of $\mathscr{F}_0(-\frac{i}{2})$ equals the weights of $\mathscr{F}_0$ plus $i$. 

\subsection*{Pure of weight zero}
Let $\mathscr{S}$ be a (finite up to isomorphism) collection of simple perverse sheaves that have weight 0. Then for any $i\in\mathbb{Z}$ and $S\in\mathscr{S}$, $S[2i](i)$ is also pure of weight 0. Define $\purex$ as a full subcategory of $\DX$ containing objects that are finite direct sums of such objects, i.e. if $M\in\purex$, then there exist $S_1, \ldots, S_N\in\mathscr{S}$ (possibly repeating) and integers $i_1, \ldots, i_N$ such that $M = S_1[2i_1](i_1)\oplus\ldots\oplus S_N[2i_N](i_N)$. We define the \textit{length} of such an object to be the number of terms in the direct sum. 
\begin{remark} We could have also defined $\purex$ to be closed under integral shift-twist. The results in this section (appropriately modified), the construction of the realization functor in Section \ref{section:realization}, and the mixed results from Section \ref{section:Mix} still hold with this modification. However, the second $t$-structure as discussed in Section \ref{section:tstr} need not exist.
\end{remark}
\begin{definition}\label{def:kosher} %We say that $\purex$ has \textit{pure morphisms} if for all simple perverse sheaves $S,S'\in \mathscr{S}$, we have that $\uHom^i(S,S')$ is pure of weight $i$ if $i$ is even and 0 if $i$ is odd. 
If $\uHom(S, S'[2n](n))^{\textup{Frob}}\cong\uHom(S, S'[2n](n))$ and $\uHom^{2n+1}(S, S')$ vanishes for all $S, S'\in\mathscr{S}$ and $n\in\mathbb{Z}$, we call $\purex$ \textit{Frobenius invariant}. If $\purex$ is Frobenius invariant, then it is easy to see that for all simple perverse sheaves $S,S'\in \mathscr{S}$, we have that $\uHom^{2n}(S,S')$ is pure of weight $2n$. 
\end{definition}

\begin{lemma}\label{shiftLemma} If $\purex$ is Frobenius invariant, then for all objects $M, N$ in $\purex$, we have that \textup{Hom(}$M$, $N[n]$\textup{)}$= 0$ for all integers $n$ with $n\neq0,1$. 
\end{lemma}

\begin{proof} Note that $M$ is pure of weight 0 and $N[n]$ is pure of weight $n$. For $n>1$, the result follows from properties of mixed perverse sheaves \cite[Proposition 5.1.15]{BBD}. Assume that $M$ and $N$ have length 1 and $n<0$. Then $M = S[2i](i)$ and $N = S'[2j](j)$ for integers $i$ and $j$ with $S, S' \in \mathscr{S}$. Of course, \[\Hom(S [2i](i), S'[2j](j)[n]) = \Hom(S[2i-2j](i-j), S'[n]).\] Thus, it suffices to consider the case with $M = S[2i](i)$ and $N=S'$. 

Note that Definition \ref{def:kosher} implies that $\textup{\underline{Hom}}^j(M, N[n])$ is pure of weight $n+j$ since \[ \uHom^j(M, N[n]) = \uHom^j(S[2i](i), S'[n]) =\uHom^{j-2i+n}(S, S')(-i).\] In particular, for $j=0,-1$, $\textup{\underline{Hom}}^j(M, N[n])$ is pure of non-zero weight. This implies that $\uHom(M, N[n])^{\textup{Frob}}$ and $\uHom^{-1}(M, N[n])_{\textup{Frob}}$ are both zero (see Remark \ref{rm:purity}). Thus, by the short exact sequence \eqref{eq:FrobSES}, we have that $\Hom(M, N[n])=0$.

For objects $M$ and $N$ in $\purex$ of lengths greater than 1, the claim follows since $\Hom$ commutes with finite direct sums. \end{proof}

The following lemmas are preparation for Section \ref{section:Mix} when we prove that a certain category $\KN$ (or more generally $\KX$) is a mixed version of $\DNbar$ (or $\DXbar$).

\begin{lemma}\label{puremorphisms} If $\purex$ is Frobenius invariant, then $\uHom^i(M, N)$ is pure of weight $i$ for $i$ even and vanishes for $i$ odd for all $M, N\in\purex$.
\end{lemma}
\begin{proof} By Definition \ref{def:kosher}, the result follows for $M$ and $N$ in $\purex$ of length 1. For arbitrary objects, the claim follows since $\uHom$ commutes with finite sums. \end{proof}

\begin{corollary}\label{lemma:Frob invariant} For all $M, N\in\purex$, we have $\Hom(M,N)\cong\uHom(M, N)^{\textup{Frob}}$ if $\purex$ is Frobenius invariant. 
\end{corollary}

\begin{proof} Let $M, N\in\purex$. Then $\uHom^{-1}(M, N)$ vanishes by Lemma \ref{puremorphisms}. Thus, the Frobenius coinvariants $\uHom^{-1}(M, N)_{\textup{Frob}}$ are also trivial. Hence, by the short exact sequence \eqref{eq:FrobSES}, we see that $\Hom(M, N)\cong\uHom(M, N)^{\textup{Frob}}$. \end{proof}

\begin{lemma}\label{all Frob invariant} Suppose that $\purex$ is Frobenius invariant. Then $\Hom(M,N)\cong\uHom(M,N)^{\textup{Frob}}\cong\uHom(M,N)$ for all $M, N\in\purex$.
\end{lemma}
\begin{proof} The claim holds for objects of length 1 by the definition of Frobenius invariant and Corollary \ref{lemma:Frob invariant}. The general case holds since $\Hom$ and $\uHom$ commute with finite direct sums. \end{proof}

\section{A Realization Functor}\label{section:realization}

In this section, we construct a triangulated functor from $\KX$ to $\DX$ assuming that $\purex$ is Frobenius invariant.
Our method is based on Beilinson's construction of a \textit{realization} functor in \cite{B}. However applying this technique in a setting without a $t$-structure is due to an idea of Achar and Kitchen. We briefly review the definition of a filtered triangulated category and some of its important properties. Note that we consider increasing filtrations, while Beilinson considers decreasing filtrations.  
\begin{definition} We say that a triangulated category $\mathcal{D}$ is \textit{filtered} if it has a collection of pairs of strictly full triangulated subcategories $(F^{\leq n}\mathcal{D}, F^{\geq n}\mathcal{D})_{n\in\mathbb{Z}}$ satisfying the following properties:
\begin{enumerate}
	\item If $M\in F^{\leq n}\mathcal{D}$ and $N\in F^{\geq n+1}\mathcal{D}$, then $\Hom(M,N)=0$.
	\item We have $F^{\leq n}\mathcal{D}\subset F^{\leq n+1}\mathcal{D}$ and $F^{\geq n}\mathcal{D}\supset F^{\geq n+1}\mathcal{D}$.
	\item For any $Z\in\mathcal{D}$ and $n\in\mathbb{Z}$, there is a distinguished triangle $A\rightarrow Z\rightarrow B\rightarrow$ with $A\in F^{\leq n-1}\mathcal{D}$ and $B\in F^{\geq n}\mathcal{D}$.
	\item The filtration is \textit{bounded}, i.e. $\bigcup_{n\in\mathbb{Z}} F^{\leq n}\mathcal{D} = \bigcup_{n\in\mathbb{Z}} F^{\geq n}\mathcal{D} = \mathcal{D}$.
	\item We have a \textit{shift of filtration} $(s,\alpha)$. Here $s:\mathcal{D}\rightarrow\mathcal{D}$ is an autoequivalence so that $s(F^{\leq n}\mathcal{D})= F^{\leq n+1}\mathcal{D}$ and $s(F^{\geq n}\mathcal{D})= F^{\geq n+1}\mathcal{D}$ and $\alpha$ is a natural transformation $s\rightarrow \id_\mathcal{D}$ with $\alpha_M=s(\alpha_{s^{-1}M})$.
	\item For all $M\in F^{\geq 1}\mathcal{D}$ and $N\in F^{\leq 0}\mathcal{D}$, the natural tranformation $\alpha$ induces isomorphisms 
\begin{align}\label{sisom}
	\Hom(M,N)\cong\Hom(M, sN)\cong\Hom(s^{-1}M, N).
\end{align}
\end{enumerate}
\end{definition} The inclusion functor $F^{\leq n}\mathcal{D}\rightarrow\mathcal{D}$ admits a right adjoint denoted $w_{\leq n}:\mathcal{D}\rightarrow F^{\leq n}\mathcal{D}$. Similarly, $F^{\geq n}\mathcal{D}\rightarrow\mathcal{D}$ admits a left adjoint denoted $w_{\geq n}:\mathcal{D}\rightarrow F^{\geq n}\mathcal{D}$. It is shown in \cite[Proposition A.3]{B} that for each $n\in\mathbb{Z}$ the distinguished triangle in (3) is canonically isomorphic to 
\begin{align}\label{filteredtriangle}
w_{\leq n-1}Z\rightarrow Z\rightarrow w_{\geq n} Z\rightarrow.	
\end{align} Let $\mathcal{D}^n = F^{\leq n}\mathcal{D}\cap F^{\geq n}\mathcal{D}$. The compositions $w_{\leq n}w_{\geq n}$ and $w_{\geq n}w_{\leq n}$ are naturally equivalent, and we denote them by $\gr_n :\mathcal{D}\rightarrow\mathcal{D}^n$. For an object $M$ in $\mathcal{D}$, we define the \textit{filtered support} of $M$ to be the smallest interval $[m,n]$ satisfying $M\in F^{\leq n}\mathcal{D}\cap F^{\geq m}\mathcal{D}$; in other words, $m=\textup{max}\{i\mid M\in F^{\geq i}\mathcal{D}\}$ and $n=\textup{min}\{i\mid M\in F^{\leq i}\mathcal{D}\}.$

\begin{lemma}\label{gralpha} Let $M$ be an object in $\mathcal{D}$. The morphism $\alpha_{sM}:sM\rightarrow M$ induced by the natural transformation $\alpha$ defined above has the property that $\gr_i(\alpha_{sM}) = 0$ for all $i\in\mathbb{Z}$.
\end{lemma}
\begin{proof} We prove the lemma by induction on the length of filtered support of $M$. If $M = \gr_n M$ for some $n$, the claim follows immediately since $\gr_n sM = s\gr_{n-1}M = 0$. To prove the general case, let $M$ have filtered support $[m,n]$ and consider the morphism of distinguished triangles induced by $\alpha$:
\begin{center}
 \begin{tikzpicture}[description/.style={fill=white,inner sep=2pt}]
\matrix (m) [matrix of math nodes, row sep=3em,
column sep=2.5em, text height=1.5ex, text depth=0.25ex]
{ w_{\leq n-1} sM & sM & w_{\geq n} M & w_{\leq n-1} sM[1]\\
  w_{\leq n-1} M & M & w_{\geq n} M & w_{\leq n-1} M[1]\\ };
\path[->,font=\scriptsize]
(m-1-1) edge (m-1-2)
edge (m-2-1)
(m-1-2) edge (m-1-3)
edge  (m-2-2)
(m-1-3) edge (m-1-4)
edge  (m-2-3)
(m-1-4) edge (m-2-4)
(m-2-1) edge  (m-2-2)
(m-2-2) edge  (m-2-3)
(m-2-3) edge  (m-2-4);
\end{tikzpicture}
\end{center}
Note that the filtered support of $w_{\leq n-1} M$ and $w_{\geq n} M$ has strictly shorter length than that of $M$. The functor $\gr_i$ is triangulated (\cite[Proposition 2.3]{AT}); thus, we have an induced morphism of distinguished triangles:
\begin{center}
 \begin{tikzpicture}[description/.style={fill=white,inner sep=2pt}]
\matrix (m) [matrix of math nodes, row sep=3em,
column sep=2.5em, text height=1.5ex, text depth=0.25ex]
{ \gr_i w_{\leq n-1} sM & \gr_i sM & \gr_i w_{\geq n}sM & \gr_i w_{\leq n-1} sM[1]\\
  \gr_i w_{\leq n-1} M & \gr_i M & \gr_i w_{\geq n} M & \gr_i w_{\leq n-1} M[1]\\ };
\path[->,font=\scriptsize]
(m-1-1) edge (m-1-2)
edge node[description] {$0$} (m-2-1)
(m-1-2) edge (m-1-3)
edge node[description] {$\gr_i\alpha_{sM}$} (m-2-2)
(m-1-3) edge (m-1-4)
edge node[description] {$0$} (m-2-3)
(m-1-4) edge (m-2-4)
(m-2-1) edge (m-2-2)
(m-2-2) edge (m-2-3)
(m-2-3) edge (m-2-4);
\end{tikzpicture}
\end{center}

If $i\leq n-1$, then $\gr_i w_{\geq n}sM = \gr_i w_{\geq n}M = 0$, so the maps $\gr_i w_{\leq n-1}sM \rightarrow \gr_i sM$ and $\gr_i w_{\leq n-1}M \rightarrow \gr_i M$ are isomorphisms. Commutativity of the above squares implies that $\gr_i\alpha_{sM}=0.$ Similar arguments prove the cases $i\geq n+1$ and $i = n$. \end{proof}

We say that a filtered category $\widetilde{\mathcal{D}}$ is a \textit{filtered version} of a triangulated category $\mathcal{D}$ if there is an equivalence $\mathcal{D}\rightarrow\widetilde{\mathcal{D}}^0$ of triangulated categories. Beilinson proves in \cite{B} the existence of a unique functor (up to unique isomorphism) $\omega: \widetilde{\mathcal{D}}\rightarrow\mathcal{D}$ satisfying the following conditions: 
\begin{enumerate}
	\item $\omega\big|_{F^{\geq 0}\widetilde{\mathcal{D}}}$ is left adjoint to the inclusion functor $\mathcal{D}\rightarrow F^{\geq 0}\widetilde{\mathcal{D}}$,
	\item $\omega\big|_{F^{\leq 0}\widetilde{\mathcal{D}}}$ is right adjoint to $\mathcal{D}\rightarrow F^{\leq 0}\widetilde{\mathcal{D}}$,
	\item  and $\omega(\alpha_M):\omega(sM)\rightarrow\omega(M)$ is an isomorphism.
\end{enumerate} We may think of $\omega$ as the functor that forgets the filtration. For $M\in F^{\geq 0}\widetilde{\mathcal{D}}$ and $N\in F^{\leq 0}\widetilde{\mathcal{D}}$, $\omega$ induces an isomorphism
\begin{align}\label{omegaiso}
	\Hom_{\widetilde{\mathcal{D}}}(M, N) \simeq \Hom_{\mathcal{D}}(\omega(M), \omega(N)).
\end{align}

From now on, denote by $\widetilde{\mathcal{D}}$ a \textit{filtered version} of $\DX$. Let $\widetilde{\mathcal{A}}$ be the full subcategory of $\widetilde{\mathcal{D}}$ consisting of objects $M$ with the property that $\gr_i M \in s^i\purex[i]$ for all $i\in\mathbb{Z}$. 
\begin{remark}\label{A-tilde facts} If $w_{\leq n}M$ and $w_{\geq n+1}M$ are both in $\widetilde{\mathcal{A}}$ for some $n\in\mathbb{Z}$, then $M\in\widetilde{\mathcal{A}}$. Also $M\in\widetilde{\mathcal{A}}$ implies that $sM[1]\in\widetilde{\mathcal{A}}$. Thus, if $f:M\rightarrow N$ is a morphism in $\widetilde{\mathcal{A}}$, then the cone of the composition \[sM\stackrel{\alpha}{\rightarrow}M\stackrel{f}{\rightarrow}N\] is also in $\widetilde{\mathcal{A}}$ since Lemma \ref{gralpha} implies the graded pieces are given by $\gr_n cone(f\circ\alpha)=\gr_n sM[1]\oplus \gr_n N$. 
\end{remark}

Here is an outline of the construction of the realization functor: first, we show $\widetilde{\mathcal{A}}$ is equivalent to $\CX$ via the functor $\beta$ (to be defined in \eqref{beta}). The composition $\omega\circ\beta^{-1}$ gives a functor from $\CX$ to $\DX$. Next, we show that this functor takes null-homotopic maps to 0 . Thus, $\omega\circ\beta^{-1}$ factors through $\KX$. 
\begin{center}
\begin{tikzpicture}[description/.style={fill=white,inner sep=2pt}]
\matrix (m) [matrix of math nodes, row sep=3em,
column sep=2.5em, text height=1.5ex, text depth=0.25ex]
{ \CX & \stackrel{\sim}{\mathcal{A}} & \stackrel{\sim}{\mathcal{D}} & \DX \\
  \KX \\ };
\path[->,font=\scriptsize]
(m-1-1) edge node[above] {$\beta^{-1}$} node[below] {$\sim$} (m-1-2)
(m-1-3) edge node[auto] {$\omega$} (m-1-4)
(m-2-1) edge node[below] {$\stackrel{\sim}{\beta}$} (m-1-4);
\path[right hook->] (m-1-2) edge (m-1-3);
\path[dashed, ->] (m-1-1) edge (m-2-1);

\end{tikzpicture}
\end{center} 

\begin{lemma}\label{redundant} Let $M$ and $N$ be objects in $\widetilde{\mathcal{A}}$. Then $\Hom(\gr_n M, w_{\leq n-1} N) = 0$ for all $n\in\mathbb{Z}$
\end{lemma}

\begin{proof} This is a direct consequence of Lemma \ref{shiftLemma}.
%\Hom(i gr_0M,w_{\leq}N)\cong\Hom(gr_0M,\w_{\geq -1}w_{\leq -1}N) \cong \Hom(gr_0 M, gr_{-1} N)\cong \Hom(\omega(gr_0 M), \omega(gr_{-1} N)) = 0
\end{proof}

Our situation differs from Beilinson's in that neither $\purex$ nor $\widetilde{\mathcal{A}}$ is the heart of a $t$-structure. In particular, we need the following lemma.

\begin{lemma}\label{faithful}
Let $f: M\rightarrow N$ be a morphism in $\widetilde{\mathcal{A}}$. Then $f=0$ if and only if $\gr_i f = 0$ for all $i\in\mathbb{Z}$.
\end{lemma}

\begin{proof}First we show that $w_{\leq n-1} f = 0$ implies that $w_{\leq n} f = 0$.  Consider the following morphism of distinguished triangles:
\begin{center}
 \begin{tikzpicture}[description/.style={fill=white,inner sep=2pt}]
\matrix (m) [matrix of math nodes, row sep=3em,
column sep=2.5em, text height=1.5ex, text depth=0.25ex]
{ w_{\leq n-1} M & w_{\leq n} M & \gr_n M & w_{\leq n-1} M[1]\\
  w_{\leq n-1} N & w_{\leq n} N & \gr_n N & w_{\leq n-1} N[1]\\ };
\path[->,font=\scriptsize]
(m-1-1) edge node[auto] {$u$} (m-1-2)
edge node[description] {$ w_{\leq n-1}f=0 $} (m-2-1)
(m-1-2) edge node[auto] {$v$} (m-1-3)
edge node[description] {$w_{\leq n}f $} (m-2-2)
(m-1-3) edge node[auto] {$w$} (m-1-4)
edge node[description] {$ \gr_n f=0$} (m-2-3)
(m-1-4) edge (m-2-4)
(m-2-1) edge node[auto] {$ u' $} (m-2-2)
(m-2-2) edge node[auto] {$ v' $} (m-2-3)
(m-2-3) edge node[auto] {$ w' $} (m-2-4);
\end{tikzpicture}
\end{center} Since the squares commute, we see that $v'w_{\leq n}f = 0$. Thus, there is a morphism $s$ in $\Hom(w_{\leq n} M, w_{\leq n-1} N)$ so that $u's = w_{\leq n}f$. Similarly, since $w_{\leq n}f u = 0$, there exists $t$ in $\Hom(\gr_n M, w_{\leq n} N)$ so that $tv = w_{\leq n}f$. This gives a morphism of distinguished triangles.

\begin{center}
 \begin{tikzpicture}[description/.style={fill=white,inner sep=2pt}]
\matrix (m) [matrix of math nodes, row sep=3em,
column sep=2.5em, text height=1.5ex, text depth=0.25ex]
{ w_{\leq n-1} M & w_{\leq n} M & \gr_n M & w_{\leq n-1} M[1]\\
  \gr_n N [-1] & w_{\leq n-1} N & w_{\leq n} N & \gr_n N \\ };
\path[->,font=\scriptsize]
(m-1-1) edge node[auto] {$u$} (m-1-2)
edge node[description] {$ h $} (m-2-1)
(m-1-2) edge node[auto] {$v$} (m-1-3)
edge node[description] {$s$} (m-2-2)
(m-1-3) edge node[auto] {$w$} (m-1-4)
edge node[description] {$t$} (m-2-3)
(m-1-4) edge node[description] {$ h[1]$} (m-2-4)
(m-2-1) edge node[auto] {$ -w' $} (m-2-2)
(m-2-2) edge node[auto] {$ u' $} (m-2-3)
(m-2-3) edge node[auto] {$ v' $} (m-2-4);
\end{tikzpicture}
\end{center} Now, we have that $h = 0$ and $h[1] = 0$ since $\Hom(w_{\leq n-1}M, \gr_n N [-1])=0$ by property (1) of the filtered derived category. Thus, we see that $v't=0$. Next, we apply the functor $\Hom(\gr_n M, - )$ to the bottom distinguished triangle to get the exact sequence
	\[
\Hom(\gr_n M, w_{\leq n-1}N)\stackrel{u'}{\longrightarrow}\Hom(\gr_n M, w_{\leq n}N)\stackrel{v'}{\longrightarrow}\Hom(\gr_n M, \gr_n N).
\] We have that $v't=0$; thus, $t\in$ Ker $v'$ = Im $u'$. However, Lemma \ref{redundant} implies that $\Hom(\gr_n M, w_{\leq n-1} N)=0$. Thus, $t=0$ and hence, $w_{\leq n} f = 0$.

Let $m=\textup{max}\{i\mid M\in F^{\geq i}\widetilde{\mathcal{D}}\}$ and $n=\textup{min}\{i\mid M\in F^{\leq i}\widetilde{\mathcal{D}}\}$. We proceed by induction on the length of the interval [$m$,$n$]. If $m=n$, then $M = \gr_n M$. In this case, $f = \gr_n f = 0$ by hypothesis. If $m<n$, then $w_{\leq m}f = \gr_m f = 0$. The above argument and induction implies that $w_{\leq n}f = 0$, but $w_{\leq n}f = f$. \end{proof}

Now we define the functor 
\begin{align}\label{beta}
	\beta:\widetilde{\mathcal{A}} \rightarrow \CX.
\end{align} Let $M$ be an object in $\widetilde{\mathcal{A}}$. Let $\beta(M)$ be the chain complex $M^\bullet$ with $M^i = \omega(\gr_{-i}M)[i]=\gr_0(s^iM)[i]$ and differential $\delta^i:M^i\rightarrow M^{i+1}$ given by the third morphism in the functorial distinguished triangle	
	\[ \omega(\gr_{-i-1}M)[i]\rightarrow\omega(w_{\geq-i-1}w_{\leq-i}M)[i]\rightarrow\omega(\gr_{-i}M)[i]\stackrel{\delta^i}{\rightarrow}\omega(\gr_{-i-1}M)[i+1].
\]
\begin{lemma} The functor $\beta$ takes $M$ to a chain complex $M^\bullet$ with differential $\delta$.
\end{lemma}
\begin{proof} It is sufficient to show that the composition $\delta^{i+1}\circ\delta^i = 0$. Consider the following commutative diagrams.
%{\scalefont{0.6}
\begin{center}
\begin{tikzpicture}
\small
\node (A) at (1.5, 0) {$\gr_{-i-1}M$};
\node (B) at (8.5, 0) {$w_{\geq -i-2}w_{\leq -i}M$}; 
\node (C) at (0, -1.5) {$w_{\geq -i-2}w_{\leq -i-1}M$}; 
\node (D) at (3,-1.5) {$w_{\geq -i-1}w_{\leq -i}M$};
\node (E) at (7, -1.5) {$w_{\geq -i-2}w_{\leq -i-1}M$};
\node (F) at (10, -1.5) {$w_{\geq -i-1}w_{\leq -i}M$};
\draw[->] (A) -- (D);
\draw[->] (C) -- (A); 
\draw[->] (C) -- (D);
\draw[->] (B)	-- (F);
\draw[->] (E) -- (B);
\draw[->] (E) -- (F);

\end{tikzpicture}
\end{center}
Let $N$ be the cone of the morphism $w_{\geq -i-2}w_{\leq -i-1}M\rightarrow w_{\geq -i-1}w_{\leq -i}M$. The octahedral axiom applied to each diagram yields two distinguished triangles
\[
 \gr_{-i-2}M[1]\rightarrow N\rightarrow \gr_{-i}M\stackrel{q}{\rightarrow}\gr_{-i-2}M[2],\]
\[		 \gr_{-i}M\rightarrow N\rightarrow \gr_{-i-2}M[1]\rightarrow\gr_{-i}M[1].\] Note that the second triangle splits since $\Hom(\gr_{-i-2}M, \gr_{-i}M)=0$. This implies that the first triangle must split as well, so $q=0$. Thus, our composition $\delta^{i+1}\circ\delta^i =0$ since  $\delta^{i+1}\circ\delta^i = \omega(q)[i]$. \end{proof}

\begin{proposition}\label{realization} Let $\widetilde{\mathcal{A}}$ be defined as above, and assume that $\purex$ is Frobenius invariant.
\begin{enumerate}
	\item The functor $\beta : \widetilde{\mathcal{A}} \rightarrow \CX$ is an equivalence of additive categories.
	\item The composition $\omega\circ\beta^{-1}:\CX\rightarrow\DX$ factors through the category $\KX$ and induces a functor $\widetilde{\beta}:\KX\rightarrow\DX$ such that the restriction
	\[\KX\big|_{\purex} : \purex\rightarrow \DX\] is isomorphic to the inclusion functor.
	%\item $\widetilde{\beta}$ is faithful: for all $M, N\in\KX$, we have that $\Hom_{\KX}(M,N)\hookrightarrow\Hom_{\DX}(\widetilde{\beta}M, \widetilde{\beta}N)$.
\end{enumerate}
\end{proposition}

\begin{proof}
To show the equivalence, we must show that $\beta$ is full, faithful, and essentially surjective. Lemma \ref{faithful} implies that $\beta$ is faithful. We prove fullness by induction on filtered support. Let $M$ and $N$ be objects in $\widetilde{\mathcal{A}}$. First, we assume that $M, N\in\widetilde{\mathscr{D}}^n$ for some $n\in\mathbb{Z}$. Then $\beta(M)$ and $\beta(N)$ are chain complexes concentrated in degree $-n$ and it follows from the isomorphism \eqref{omegaiso} that we have an isomorphism
\[ \Hom_{\widetilde{\mathcal{D}}}(M, N) \simeq \Hom_{\CX}(\beta(M), \beta(N)).\] Now suppose that $M, N\in F^{\geq m}\widetilde{\mathcal{D}}\cap F^{\leq n}\widetilde{\mathcal{D}}.$ We consider the truncations given by
\[ M' = w_{\leq n-1}M,\hspace{.7cm} M'' = w_{\geq n}M,\hspace{.7cm} N' = w_{\leq n-1}N,\hspace{.7cm} N'' = w_{\geq n}N.\] %Note that as a chain complex, $\beta(M')$ is the same as $\beta(M)$ except in degree $-n$ 
Let $q:\beta(M'')[-1]\rightarrow\beta(M')$ be the chain map induced by the differential $\delta^{-n}:\beta(M)^{-n}\rightarrow\beta(M)^{-n+1}$. We will need to make use of a lift $\tilde{q}$ of $q$ to $\widetilde{\mathscr{D}}$. Note that $\beta(M'')[-1] \cong \beta(s^{-1}M''[-1])$. Let $\tilde{q}: s^{-1}M''[-1]\rightarrow M'$ be the natural map obtained by applying the isomorphism \eqref{sisom} to the first map of the distinguished triangle $M''[-1]\rightarrow M'\rightarrow M\rightarrow M''$. It is easy to see that $\beta(\tilde{q}) = q$. 

Let $f\in\Hom_{\CX}(\beta(M), \beta(N))$, and let $f':\beta(M')\rightarrow\beta(N')$ and $f'':\beta(M'')\rightarrow\beta(N'')$ be the induced chain maps. The diagram
\begin{center}
 \begin{tikzpicture}[description/.style={fill=white,inner sep=2pt}]
\matrix (m) [matrix of math nodes, row sep=3em,
column sep=2.5em, text height=1.5ex, text depth=0.25ex]
{ \beta(M'')[-1] & \beta(M')\\
  \beta(N'')[-1] & \beta(N')\\ };
\path[->,font=\scriptsize]
(m-1-1) edge node[auto] {$q$} (m-1-2)
edge node[auto] {$ f''[-1] $} (m-2-1)
(m-1-2) edge node[auto] {$f'$} (m-2-2)
(m-2-1) edge node[auto] {$ q $} (m-2-2);
\end{tikzpicture}
\end{center}
commutes since $f$ is a chain map. By induction, there are morphisms $\tilde{f}':M'\rightarrow N'$ and $\tilde{f}'':s^{-1}M''[-1]\rightarrow s^{-1}N''[-1]$ such that $\beta(\tilde{f}') = f'$ and $\beta(\tilde{f}'') = f''[-1]$. Consider the diagram
\begin{center}
 \begin{tikzpicture}[description/.style={fill=white,inner sep=2pt}]
\matrix (m) [matrix of math nodes, row sep=3em,
column sep=2.5em, text height=1.5ex, text depth=0.25ex]
{ M''[-1] & s^{-1}M''[-1]& M'\\
  N''[-1] & s^{-1}N''[-1]& N' \\ };
\path[->,font=\scriptsize]
(m-1-1) edge node[auto] {$\alpha$} (m-1-2)
edge node[auto] {$ s\tilde{f}'' $} (m-2-1)
(m-1-2) edge node[auto] {$\tilde{q}$} (m-1-3)
edge node[auto] {$\tilde{f}''$} (m-2-2)
(m-1-3) edge node[auto] {$\tilde{f}'$} (m-2-3)
(m-2-1) edge node[auto] {$ \alpha $} (m-2-2)
(m-2-2) edge node[auto] {$ \tilde{q} $} (m-2-3);
\end{tikzpicture}
\end{center}
Since $\alpha$ is a natural transformation, we see that the left-hand square commutes. The right-hand square commutes because $\beta$ is faithful. Thus, the outer square commutes as well, so we may complete it to a morphism of distinguished triangles
\begin{center}
 \begin{tikzpicture}[description/.style={fill=white,inner sep=2pt}]
\matrix (m) [matrix of math nodes, row sep=3em,
column sep=2.5em, text height=1.5ex, text depth=0.25ex]
{ M''[-1] & M' & M & M''\\
  N''[-1] & N' & N & N''\\};
\path[->,font=\scriptsize]
(m-1-1) edge (m-1-2)
edge node[left] {$ s\tilde{f}'' $} (m-2-1)
(m-1-2) edge (m-1-3)
edge node[left] {$\tilde{f}'$} (m-2-2)
(m-1-3) edge (m-1-4)
(m-2-1) edge (m-2-2)
(m-2-2) edge (m-2-3)
(m-2-3) edge (m-2-4)
(m-1-4) edge (m-2-4);
\path[dashed, ->] (m-1-3) edge node[right] {$\tilde{f}$} (m-2-3);
\end{tikzpicture}
\end{center} We have $\beta(\tilde{f}) = f$, so $\beta$ is full.

A similar argument proves that $\beta$ is essentially surjective. It is easy to see that any chain complex concentrated in a single degree is in the image of $\beta$. Now, if $M^\bullet\in \CX$ such that $M^i = 0$ except when $-n\leq i \leq -m$, then the differential $\delta^{-n}$ induces a chain map $q: M''[-1]\rightarrow M'$, where $M''$ is concentrated in degree $-n$ and $M'$ vanishes  except in degrees $-n+1, \ldots, -m$. By induction, we have objects $\tilde{M}'', \tilde{M}'\in\tilde{\mathscr{A}}$ so that $\beta(\tilde{M}'') = M''$ and $\beta(\tilde{M}') = M'$. Since $\beta$ is fully faithful, we have a morphism $\tilde{q}: \tilde{M}''\rightarrow \tilde{M}'$ with $\beta(\tilde{q}) = q$. Let $\tilde{M}$ be the cone of the morphism $\tilde{q}\circ\alpha:s\tilde{M}''\rightarrow\tilde{M}'$. Then $\beta(\tilde{M})\cong M$. 

Now, we consider part (2). Let $f:M^\bullet\rightarrow N^\bullet$ be a morphism in $\CX$, corresponding via $\beta$  to $\tilde{f}:\tilde{M}\rightarrow\tilde{N}$. Let $Z^\bullet$ denote the cone of $f$, and let $\tilde{Z}$ denote the cone of the composition \[s\tilde{M}\stackrel{\alpha}{\rightarrow}\tilde{M}\stackrel{\tilde{f}}{\rightarrow}\tilde{N}.\] Note that $\tilde{Z}\in\widetilde{\mathcal{A}}$ by Remark \ref{A-tilde facts}. Since $\omega(sM)\cong\omega(M)$, we see that $\omega\circ\beta^{-1}$ takes the diagram $M^\bullet\rightarrow N^\bullet\rightarrow Z^\bullet \rightarrow M^\bullet [1]$ to a distinguished triangle in $\DX$. If $f$ is null-homotopic, then the homotopy induces a chain map $Z^\bullet\rightarrow N^\bullet$ which induces a splitting of the triangle \[\tilde{N}\rightarrow\tilde{Z}\rightarrow s\tilde{M}[1]\stackrel{\tilde{f}\circ\alpha[1]}{\rightarrow}\tilde{N}[1]\] in $\widetilde{\mathcal{D}}$.  Thus, $\omega\circ\beta^{-1}(f)=0$. \end{proof}
\begin{remark}We assumed that $\purex$ was Frobenius invariant. However, the above construction still holds if we replace \textit{Frobenius invariance} with the weaker condition that $\uHom^i(S, S')$ is pure of weight $i$ for all $i\in\mathbb{Z}$ and $S, S'\in\mathscr{S}$.
\end{remark}

\section{Mixedness}
\label{section:Mix}

\subsection{Mixed and Orlov categories.}
Let $\mathscr{M}$ be a finite-length abelian category. As in \cite[Definition 4.1.1]{BGS}, a \textit{mixed structure} on $\mathscr{M}$ is a function $\textup{wt}:\textup{Irr}(\mathscr{M})\rightarrow\mathbb{Z}$ such that
\begin{equation}
\label{eq:mixed}
\textup{Ext}^1(S,S')=0 \textup{ if } S, S' \textup{ are simple objects with wt}(S)\leq \textup{wt}(S').
\end{equation}
As in \cite[Section 2.2]{AR}, we can extend the notion of a mixed structure to a triangulated category in the following way. Let $\mathscr{D}$ be a triangulated category with a bounded $t$-structure whose heart is $\mathscr{M}$. A \textit{mixed structure} on $\mathscr{D}$ is a mixed structure on $\mathscr{M}$ satisfying
\begin{equation}
\label{eq:mixedtriang}
\Hom_{\mathscr{D}}^i(S,S')=0 \textup{ if } S, S'\in\mathscr{M} \textup{ are simple objects with wt}(S)< \textup{wt}(S') +i.
\end{equation}

Let $\mathscr{A}$ be an additive category and $\textup{Ind}(\mathscr{A})$ be the set of isomorphism classes of indecomposable objects in $\mathscr{A}$. The category $\mathscr{A}$, equipped with a function $\textup{deg}: \textup{Ind}(\mathscr{A})\rightarrow\mathbb{Z}$, is called an \textit{Orlov category} (see \cite[Definition 4.1]{AR}) if the following conditions hold:
\begin{enumerate}
	\item All Hom-spaces in $\mathscr{A}$ are finite-dimensional.
	\item For any $S\in \textup{Ind}(\mathscr{A})$, we have $\textup{End}(S)\cong\Ql$.
	\item If $S,S'\in \textup{Ind}(\mathscr{A})$ with $\deg(S)\leq\deg(S')$ and $S\not\cong S'$, then Hom($S$,$S'$)=0.
\end{enumerate} According to \cite[Proposition 5.4]{AR}, the homotopy category of an Orlov category $\orlov$ has a natural $t$-structure whose heart is a finite-length abelian category containing irreducibles given by $A[\deg(A)]$ for $A\in\textup{Ind}(\mathscr{A}).$ Also, the function $\textup{wt}(A[\deg(A)]) = \deg(A)$ makes $\orlov$ into a mixed category. 

\begin{remark}The category $\purex$ is Orlov. An indecomposable object in $\purex$ is given by $S[2m](m)$ where $S$ is a simple perverse sheaf in $\mathscr{S}$. We define the degree function by deg($S[2m](m)$)=$-2m$. To see that this degree function makes $\purex$ into an Orlov category, we simply note that for $S, S'\in\mathscr{S}$ with $S[2m](m)\not\cong S'[2n](n)$
	\[ \Hom_{\purex}(S[2m](m), S'[2n](n)) = \Hom^{2n-2m}_{\DX}(S(m), S'(n)).
\]
When $-2m<-2n$, $2n-2m$ is negative, and this vanishes since $S(m)$ and $S'(n)$ are objects in the heart of a $t$-structure on $\DX$. If $-2m=-2n$, this vanishes since we assume that $S[2m](m)\not\cong S'[2n](n)$ implying that $S$ and $S'$ are nonisomorphic simple objects. \end{remark}
 
We will denote the heart of $\KX$ by $\PMX$. The simple objects in $\PMX$ are given by $(S[2i](i))[-2i]$ for any $S\in\mathscr{S}, i\in\mathbb{Z}$. Note that the two shifts do not cancel since they occur in different triangulated categories. By \cite{AR}, the category $\PMX$ is a mixed category with weight function $\textup{wt}(S[2i](i)[-2i]) = -2i$ and a degree 2 Tate twist $\langle 2\rangle:=[-2](-1)[2]$. In the remainder of this section, we will show that $\KX$ is a mixed version of its analogue over $\bar{\mathbb{F}}_q$. This is defined as follows:

\begin{definition} Let $\mathscr{D}$ and $\mathscr{D}'$ be $t$-categories such that 
\begin{itemize}
	\item $\mathscr{D}$ is a mixed triangulated category with a $t$-exact autoequivalence (\textit{a degree d Tate twist}) $\langle d\rangle:\mathscr{D}\rightarrow\mathscr{D}$ satisfying $\textup{wt}(S\langle d\rangle) = \textup{wt}(S) + d$;
	\item there is a $t$-exact functor $F:\mathscr{D}\rightarrow\mathscr{D}'$ such that the essential image generates $\mathscr{D}'$ as a triangulated category;
	\item and there is an isomorphism $\varepsilon: F\circ\langle d\rangle\stackrel{\sim}{\rightarrow}F$.
\end{itemize}
Then $\mathscr{D}$ is called a \textit{mixed version} of $\mathscr{D}'$ if $\varepsilon$ induces an isomorphism for all objects $M, N\in\mathscr{D}$ 
\begin{equation}
\displaystyle\bigoplus_{n\in\mathbb{Z}}\Hom_{\mathscr{D}}(M,N\langle nd\rangle) \xrightarrow{\sim}\Hom_{\mathscr{D'}}(FM, FN). 
\label{eq:mixed version}
\end{equation}
\end{definition}

Let $\DXc$ be the $\ell$-adic derived category of complexes of sheaves on $X:= X_0\times_{\textup{Spec}\mathbb{F}_q}\textup{Spec}(\bar{\mathbb{F}}_q)$. Define $\mathcal{F}: \KX\rightarrow\DXc$ to be the functor given by the composition \[\KX\stackrel{\tilde{\beta}}{\rightarrow}\DX\stackrel{\xi}{\rightarrow}\DXc.\] Recall that $\KX\stackrel{\tilde{\beta}}{\rightarrow}\DX$ is the realization functor defined in Section \ref{section:realization}. The functor $\DX\stackrel{\xi}{\rightarrow}\DXc$ is given by extending scalars to $\bar{\mathbb{F}}_q$. Let $\DXbar$ be the triangulated category generated in $\DXc$ by the objects $\mathcal{F}(S)$ for $S\in\mathscr{S}$, and let $\Pervs(X)$ be the Serre subcategory of $\Perv(X)$ generated by the perverse sheaves $\mathcal{F}(S)$ for $S\in\mathscr{S}$. Note that $\DXbar$ contains the image of $\mathcal{F}$ in $\DXc$. Thus, we may think of $\mathcal{F}$ as a functor with target $\DXbar$. Note that shifts in $\KX$ and $\DX$ combine under $\mathcal{F}$. Thus, for $M\in\KX$, we have that $\mathcal{F}(M[2i](i)[-2i])\cong\mathcal{F}(M)$. Also, $\mathcal{F}$ commutes with shifts so $\mathcal{F}(M[i])\cong\mathcal{F}(M)[i]$. We now show that $\KX$ is a mixed version of $\DXbar$ assuming $\purex$ is Frobenius invariant on morphisms. 

\begin{theorem}\label{theorem:genmixedversion} Assume that $\purex$ is Frobenius invariant. Then $\KX$ is a mixed version of $\DXbar$, where $\DXbar$ is the triangulated category generated by the image of $\mathscr{S}$ in $\DXc$. 
\end{theorem}

\begin{proof} Let $M$ and $N\in\KX$. We proceed by double induction on the lengths of the chain complexes $M$ and $N$. First, assume that $M$ and $N$ are concentrated in one degree. Without loss of generality, assume $M$ is concentrated in degree 0, i.e. that $M\in\Pure_\mathscr{S}(X_0)$. Let $j\in\mathbb{Z}$ be such that $N[-j]\in\purex$. Then $\Hom(M, N\langle 2n\rangle)) = \Hom(M, N[-2n](-n)[2n])\neq 0$ implies that $2n = -j$ because otherwise, $M$ and $N\langle 2n\rangle$ would be chain complexes concentrated in different degrees. Now, if $j$ is odd, then $\displaystyle\oplus_{n\in\mathbb{Z}}\Hom(M,N\langle 2n\rangle)=0.$ In this case, we must show that $\Hom(\mathcal{F}M, \mathcal{F}N) = 0$. Recall that $\mathcal{F}$ commutes with shift and that $N[-j]\in\purex$. Thus, we see that \[\Hom(\mathcal{F}M,\mathcal{F}N) = \Hom(\mathcal{F}M, \mathcal{F}(N[-j])[j]) = \Hom^j(\mathcal{F}M, \mathcal{F}(N[-j])).\] This vanishes by Lemma \ref{puremorphisms}. 

Now assume that $j$ is even.  

\begin{align*}
	\displaystyle\bigoplus_{n\in\mathbb{Z}}\Hom(M,N\langle 2n\rangle)&=\Hom(M, N[j](\textstyle\frac{j}{2})[-j])\\
																&\cong\Hom_{\DX}(M, (N[-j])(\textstyle\frac{j}{2})[j])\\
																&\cong\Hom_{\DXbar}(\mathcal{F}M, \mathcal{F}N), \hspace{1cm} \textup{by Lemma \ref{all Frob invariant}}
\end{align*}
Suppose that the theorem holds for $M$ a chain complex of length less than $n+1$ and $N$ concentrated in a single degree. Now, assume that $M^\bullet\in\KX$ is a chain complex of length $n+1$ and that $N$ is a chain complex concentrated in one degree. Let $i\in\mathbb{Z}$ be such that $M^\bullet$ vanishes in degrees less than $i$ and more than $i+n$. Note that the differential $\delta^i$ induces a chain map $M''[-1]\rightarrow M'$ where $M''$ and $M'$ are the obvious truncations of $M$. This gives a distinguished triangle $M'\rightarrow M^\bullet\rightarrow M''\rightarrow$. This triangle gives us the following commutative diagram with exact rows.

\begin{flushleft}
 %\scalefont{0.8}
 \begin{tikzpicture}[description/.style={fill=white,inner sep=1.5pt}] %[scale=.2]
 %\tikzstyle{every node}=[font=\small]

\matrix (m) [matrix of math nodes, row sep=2.5em,
column sep=1em, text height=1ex, text depth=0.25ex, nodes in empty cells]
{ \displaystyle\bigoplus_{n\in\mathbb{Z}}\Hom^{-1}(M', N\langle2n\rangle)  & \displaystyle\bigoplus_{n\in\mathbb{Z}}\Hom(M'', N\langle2n\rangle) & \displaystyle\bigoplus_{n\in\mathbb{Z}}\Hom(M^\bullet, N\langle2n\rangle) & \\
  \Hom^{-1}(\mathcal{F}(M'), \mathcal{F}(N)) & \Hom(\mathcal{F}(M''), \mathcal{F}(N)) & \Hom(\mathcal{F}(M^\bullet), \mathcal{F}(N)) & \\ };
\path[->,font=\scriptsize]
(m-1-1) edge (m-1-2)
edge node[above, sloped]{$\sim$} node[left] {$\alpha_1 $} (m-2-1)
(m-1-2) edge (m-1-3)
edge node[above, sloped]{$\sim$} node[left]{$\alpha_2$} (m-2-2)
(m-1-3) edge node[auto] {$u$} (m-1-4)
edge node[right] {$f$} (m-2-3)
(m-2-1) edge (m-2-2)
(m-2-2) edge (m-2-3)
(m-2-3) edge node[auto] {$u'$} (m-2-4);
\end{tikzpicture}

\end{flushleft}

\begin{flushright}
 %\scalefont{0.8}
 \begin{tikzpicture}[description/.style={fill=white,inner sep=1.5pt}] %[scale=.2]
 %\tikzstyle{every node}=[font=\small]

\matrix (m) [matrix of math nodes, row sep=2.5em,
column sep=1em, text height=1ex, text depth=0.25ex, nodes in empty cells]
{ & \displaystyle\bigoplus_{n\in\mathbb{Z}}\Hom(M', N(2n)) & \displaystyle\bigoplus_{n\in\mathbb{Z}}\Hom^1(M'', N(2n)) \\
  & \Hom(\mathcal{F}(M'), \mathcal{F}(N)) & \Hom^1(\mathcal{F}(M''), \mathcal{F}(N)) \\ };
\path[->,font=\scriptsize]
(m-1-1) edge node[auto] {$u$} (m-1-2)
%edge node[description] {$ h $} (m-2-1)
(m-1-2) edge (m-1-3)
				edge node[below, sloped]{$\sim$} node[auto]{$\alpha_3$} (m-2-2)
(m-2-1) edge node[auto]{$u'$} (m-2-2)
(m-2-2) edge (m-2-3)
(m-1-3)	edge node[below, sloped]{$\sim$} node[auto]{$\alpha_4$} (m-2-3);
\end{tikzpicture}

\end{flushright}

Note that $\alpha_1, \alpha_2, \alpha_3,$ and $\alpha_4$ are isomorphisms by the induction hypothesis. Thus, the five lemma implies that $f$ is also an isomorphism.

A similar argument proves the claim for general $M$ and $N$ in $\KX$. \end{proof}

\begin{corollary} The heart $\PMX$ of the $t$-structure on $\KX$ is a mixed version of the category of perverse sheaves $\Pervs(X)$.
\end{corollary}

\subsection{A second $t$-structure}\label{section:tstr} We now define a new $t$-structure on $\KX$, which we'll refer to as the non-standard $t$-structure. To do so, we regard $\purex$ as an Orlov category with a different degree function. Recall that an indecomposable in $\purex$ is given by $S[2i](i)$ for a simple perverse sheaf $S\in\mathscr{S}$ and $i\in\mathbb{Z}$. We define the new degree function on indecomposables by $\deg(S[2i](i)) = -i$. The same argument used eariler verifies that this degree function also makes $\purex$ into an Orlov category. Thus, \cite[Proposition 5.4]{AR} implies that we have a $t$-structure on $\KX$. We will denote the heart of this second $t$-structure by $\pervk(X_0)$. It is also mixed with a degree 1 Tate twist $\tw{1}:=[-2](-1)[1]$ and contains irreducible objects $S\tw{i}$ for all $S\in\mathscr{S}$ and $i\in\mathbb{Z}$.  

\section{Background on the Springer correspondence}\label{section:Springer} Let $G$ be a connected, reductive algebraic group defined over $\bar{\mathbb{F}}_q$. Let $\mathcal{N}\subset\mathfrak{g}$ be the variety of nilpotent elements in the Lie algebra of $G$. Our goal is to understand the representation theory of the Weyl group $W$ for $G$ by studying the Springer sheaf $\mathbf{A}$.

\subsection{The Springer sheaf} Let $\mu:\accentset{\sim}{\mathcal{N}}\rightarrow\mathcal{N}$ be the Springer resolution. Then the Springer sheaf $\mathbf{A}\in\DNc$ is defined by 
\[ \mathbf{A}:=\mu_{*}(\constant_{\widetilde{\mathcal{N}}})[d](\textstyle\frac{d}{2})\]
where $\constant_{\widetilde{\mathcal{N}}}$ is the constant sheaf on $\accentset{\sim}{\mathcal{N}}$ and $d=\dim \accentset{\sim}{\mathcal{N}}$.
Let $\DNGc$ be the $G$-equivariant derived category. For background on the equivariant derived category, see \cite{BL}. The Springer sheaf $\mathbf{A}$ is also a $G$-equivariant perverse sheaf. For most of what follows, we will consider both non-equivariant and $G$-equivariant versions of statements. The proofs in both cases are essentially the same. We will not distinguish by notation objects that belong to both $\DNc$ and $\DNGc$. The following fact is well-known and a consequence of \cite[Proposition 4.2.5]{BBD}: for $G$-equivariant perverse sheaves $\mathscr{F}$ and $\mathscr{G}$ on a $G$ variety $X$ with $G$ connected, we have that $\Hom_{\DXc}(\mathscr{F}, \mathscr{G})\cong\Hom_{\DXGc}(\mathscr{F}, \mathscr{G})$. A proof can be found in \cite[Section 1.16]{L2}.

The Springer sheaf $\mathbf{A}$ has a natural action of the Weyl group $\sigma: W\rightarrow\Aut(\mathbf{A})$ defined by Lusztig in \cite{L3}. According to \cite[Theorem 3]{BM1}, we have an isomorphism
\[ \Ql[W] \cong \End(\mathbf{A}). \] Let $\mathscr{B}$ be the variety of Borel subgroups of $G$. Then we also have an action $\kappa: \Ql[W]\rightarrow\End(\coflag)$ induced by the $W$-action on $G/T$ where $T$ is a maximal torus. We have $G$-equivariant analogues of the above. The structure of the cohomology ring of the flag variety is well understood. There is a degree doubling isomorphism of graded rings between (1) the algebra of $W$-coinvariants $\Coinv$ and $\coflag$, and (2) the symmetric algebra based on the Cartan $\Sh$ and $\coGflag$. We will often make use of the categories $\DBc$ and $\DBGc$ of $\ell$-adic sheaves on $\mathscr{B}$ constructible with respect to the trivial stratification.

We will usually focus our attentions on the category $\DNbar$ (respectively, $\DNGbar$) defined as the triangulated category generated by the simple summands of $\mathbf{A}$ in $\DNc$ (respectively, in $\DNGc$). 

\subsection{Borel--Moore homology of the Steinberg variety} Another approach to studying the relationship between $W$-representations and the Springer sheaf $\mathbf{A}$ involves the Borel--Moore homology of the Steinberg variety $Z=\accentset{\sim}{\mathcal{N}}\times_{\mathcal{N}}\accentset{\sim}{\mathcal{N}}$. For our purposes, we use the definition of the Borel--Moore homology in terms of the hypercohomology of the dualizing complex $\omega_Z$ as developed in \cite[8.3.7]{CG}.
\[ \cH^{BM}_i(Z) : = \cH^{-i}(Z, \omega_Z).\]
Many details of the relationship between $\cH^{BM}_\bullet(Z)$ and the Springer sheaf $\mathbf{A}$ are developed in \cite{CG}. In particular, we have an isomorphism of graded rings
\begin{align}\label{BMisom}
	\cH^{BM}_{2d-\bullet}(Z)\cong\Hom^\bullet(\mathbf{A}, \mathbf{A}).
\end{align}

Let $X_0$ be a variety defined over $\mathbb{F}_q$. Then we may define Borel--Moore homology with an $n$-twist (with $\ell$-adic coefficients) by \[\cH^{BM}_i(X_0, n) = \Hom_{\DX}^{-i}(\constant_{X_0}, \omega_{X_0}(n)).\] Extending scalars to $\bar{\mathbb{F}}_q$, we get the usual Borel--Moore homology (with an $n$-twist), and it has inherited an action of Frobenius. To denote this, we write $\uH^{BM}_i(X_0, n) = \uHom^{-i}(\constant_{X_0}, \omega_{X_0}(n))$. Applying the short exact sequence \eqref{eq:FrobSES}, we get the following short exact sequence relating the Borel--Moore homology groups of $X_0$ over $\mathbb{F}_q$ and $X$ over $\bar{\mathbb{F}}_q$:
\[0\rightarrow \uH^{BM}_{i+1}(X_0, n)_{\textup{Frob}}\rightarrow \cH^{BM}_{i}(X_0, n)\rightarrow\uH^{BM}_{i}(X_0, n)^{\textup{Frob}}\rightarrow 0\]
In particular, we have the following theorem:

\begin{theorem}\label{BMvanish} Let $X_0$ be a variety defined over $\mathbb{F}_q$. Then the Borel--Moore homology of $X_0$ vanishes for $i<-1$ and $i>2\dim(X_0)$. In particular, $\cH^{BM}_{2\dim(X_0)}(X_0, n)\cong\uH^{BM}_{2\dim(X_0)}(X_0, n)^{\textup{Frob}}$.
\end{theorem}

\begin{proof} This follows from the fact that $\uH^{BM}_i(X_0, n)$ is only non-zero for $0\leq i\leq 2\dim(X_0)$ and the above short exact sequence. \end{proof}

\subsection{Mixed Springer sheaves} Now, in order to apply the mixedness machinery of \cite{BBD}, we need analogues of $\mathcal{N}, \mathfrak{g}, G, \mathscr{B}$ and other related varieties defined over a finite field $\mathbb{F}_q$ of characteristic $p$. We make the following assumptions before proceeding.
\begin{enumerate}
	\item These varieties are equipped with an $\mathbb{F}_q$-rational structure.
	\item The field $\mathbb{F}_q$ is \textit{big enough} and has good characteristic.
	\item The reductive group $G$ is $\mathbb{F}_q$-split.
\end{enumerate}
The reason for our first assumption is clear. Let $\mathcal{N}_0$, $G_0$, et cetera denote $\mathbb{F}_q$-schemes whose extension of scalars are $\mathcal{N}$, $G$, et cetera.  Now, $G_0$ acts on $\mathcal{N}_0$ by the adjoint action; however, it may be the case that not all $G$-orbits of $\mathcal{N}$ appear. Since there are only finitely many nilpotent orbits, we can ensure that all of them are defined by taking a finite field extension of $\mathbb{F}_q$ if necessary. This is the reasoning behind our second assumption. We assume that $\mathbb{F}_q$ is \textit{big enough} so that the Frobenius fixed point set of each nilpotent $G$-orbit is non-empty. Our final requirement is that we must assume that $G$ is $\mathbb{F}_q$-split. In other words, $G_0$ has a split maximal torus of the same dimension as a maximal torus in $G$. We now show that in this setting, we do not lose information by passing to the $\mathbb{F}_q$-setting.

\begin{lemma}\label{BMFrob} The top degree Borel--Moore homology $\uH^{BM}_{2d}(Z, -d)$ has a basis that is Frobenius invariant. 
\end{lemma}

\begin{proof} Recall that the irreducible components of the Steinberg variety are given by closures of conormal bundles $T^*_{Y_w}(\mathscr{B}\times\mathscr{B}),$ where $Y_w$ is the $G$-orbit of $\mathscr{B}\times\mathscr{B}$ corresponding to $w\in W$. Each of these is defined over $\mathbb{F}_q$ since for split $G$, the Bruhat decomposition is defined over $\mathbb{F}_q$.

The top-degree Borel--Moore homology of an algebraic variety has a basis given by fundamental classes associated to the top-dimensional irreducible components. In order to show that $\uH^{BM}_{2d}(Z, -d)$ is Frobenius invariant, it suffices to show that our basis can be defined over $\mathbb{F}_q$. Let $X$ be an irreducible component of $Z$ such that $X_0$ is the corresponding irreducible component of $Z_0$ over $\mathbb{F}_q$. Then the fundamental class associated to $X$ is defined over $\mathbb{F}_q$. To see this, we need a canonical element $f\in\Hom(\constant_{X_0}, \omega_{X_0}[-2d](-d))=\cH^{-2d}(X_0, \omega_{X_0}(-d)) = \cH^{BM}_{2d}(X_0, -d).$ Let $U_0\subset X_0$ be a smooth dense open set and set $F_0 = X_0\backslash U_0$. Then we have inclusions \[U_0\stackrel{j}{\hookrightarrow}X_0\stackrel{i}{\hookleftarrow}F_0.\] This gives a distinguished triangle in $\DX$ \[i_*i^!\omega_{X_0}[-2d](-d)\rightarrow\omega_{X_0}[-2d](-d)\rightarrow j_*j^*\omega_{X_0}[-2d](-d)\rightarrow i_*i^!\omega_{X_0}[-2d+1](-d).\] Now we apply $\Hom(\constant_{X_0}, -)$ to get the exact sequence in Borel--Moore homology
\[\cH^{BM}_{2d}(F_0, -d)\rightarrow \cH^{BM}_{2d}(X_0, -d)\rightarrow \cH^{BM}_{2d}(U_0, -d)\rightarrow \cH^{BM}_{2d+1}(F_0, -d)\] Note that $\dim F_0<\dim X_0$. Thus, Lemma \ref{BMvanish} implies that \[\cH^{BM}_{2d}(F_0, -d)=\cH^{BM}_{2d+1}(F_0, -d)= 0\] since both are Borel--Moore homology groups in degree greater than $2\dim F_0$. Therefore we have an isomorphism $\cH^{BM}_{2d}(X_0, -d)\cong\cH^{BM}_{2d}(U_0, -d),$ i.e. an isomorphism \[\Hom(\constant_{X_0}, \omega_{X_0}[-2d](-d))\cong\Hom(\constant_{U_0}, \constant_{U_0}).\] Let $f$ be the morphism corresponding to $id:\constant_{U_0}\rightarrow\constant_{U_0}$ under this isomorphism.\end{proof}

We use the Borel--Moore homology as a stepping stone to get the following result.

\begin{lemma}\label{lemma:trivialFrob} We have an isomorphism $\Hom_{\DN}(\mathbf{A},\mathbf{A})\cong \uHom(\mathbf{A}, \mathbf{A})\cong\Ql[W]$. Thus, we also have $\Hom_{\DNG}(\mathbf{A},\mathbf{A})\cong \Ql[W]$.
\end{lemma}
\begin{proof}

Recall the short exact sequence \[0\rightarrow \uHom^{-1}(\mathbf{A},\mathbf{A})_{\textup{Frob}}\rightarrow \Hom(\mathbf{A},\mathbf{A})\rightarrow \uHom(\mathbf{A}, \mathbf{A})^{\textup{Frob}}\rightarrow 0.\] Since $\mathbf{A}$ is a perverse sheaf, $\uHom^{-1}(\mathbf{A},\mathbf{A})=0$ implying the map $\Hom(\mathbf{A},\mathbf{A})\rightarrow \uHom(\mathbf{A}, \mathbf{A})^{\textup{Frob}}$ is an isomorphism. Thus, we have an injection $\Hom(\mathbf{A},\mathbf{A})\hookrightarrow \uHom(\mathbf{A}, \mathbf{A})$, so it suffices to show that $\Hom(\mathbf{A},\mathbf{A})$ and $\uHom(\mathbf{A}, \mathbf{A})$ have the same dimension. Recall the isomorphism \eqref{BMisom} from \cite{CG}. This restricts to an isomorphism of the degree 0 piece: $\uHom(\mathbf{A}, \mathbf{A})\cong \uH^{BM}_{2d}(Z_0, -d)$. The following string of isomorphisms gives us the conclusion:

\begin{align*}\uHom(\mathbf{A}, \mathbf{A})^{\textup{Frob}}&\cong \uH^{BM}_{2d}(Z_0, -d)^{\textup{Frob}}\\
						&\cong \uH^{BM}_{2d}(Z_0, -d), \hspace{1cm}\textup{by Lemma \ref{BMFrob}}\\
						&\cong \uHom(\mathbf{A}, \mathbf{A})
\end{align*} \end{proof}

Thus, we have a natural action of the Weyl group $W$ on $\mathbf{A}\in\DN$. Let $\sigma: W\rightarrow \Aut(\mathbf{A})$ denote this action. By \cite[Proposition 5.3.9]{BBD}, we may decompose $\mathbf{A}$ as follows:
	\[\mathbf{A} \simeq \displaystyle\bigoplus_{\chi\in Irr(W)} \IC_\chi \otimes V_\chi,
\]
where the $\IC_\chi$ are various distinct simple perverse sheaves and $V_\chi$ is a vector space with an action of Frobenius, which is not necessarily semisimple. Note that both the Frobenius action and the Weyl group action on $V_\chi$ arise via the identification: $V_\chi = \uHom(\IC_\chi, \mathbf{A})$. However this identification requires a choice of $\IC_\chi\in\DN$. Regardless of the choice of $\IC_\chi$, it is easy to see that Frobenius is a $W$-equivariant endomorphism of $V_\chi$ by Lemma \ref{lemma:trivialFrob}. Of course $V_\chi$ is an irreducible $W$-representation. Hence, Frobenius must act by a scalar (with absolute value 1). Since a scalar may travel across the tensor product $\IC_\chi\otimes V_\chi$, we declare that Frobenius acts trivially on $V_\chi$, and this gives a unique choice for the simple perverse sheaf $\IC_\chi\in \DN$. We fix this choice for all $\chi\in\textup{Irr}(W)$.  

Define $\spr$ as the full subcategory of $\DN$ (or $\DNG$) consisting of objects that are finite direct sums of the simple perverse sheaves $\IC_\chi$ as above. Let $\puren$, respectively $\puregn$, be the full subcategory of $\DN$, respectively $\DNG$, consisting of semisimple objects that are pure of weight $0$ and whose length 1 subquotients are in $\spr[2n](n)$ for some $n$. In other words, \[\puren = \displaystyle\bigoplus_{n\in\mathbb{Z}}\spr[2n](n).\] In the following section, we will show that the categories $\puren$ and $\puregn$ are Frobenius invariant by relating morphisms between simple perverse sheaves with the cohomology (or $G$-equivariant cohomology) of the flag variety $\mathscr{B}$. To do so, we must first introduce a pair of functors. 

\subsection{The functors $\Psi$ and $\Phi$}\label{adjointpair} We now introduce an adjoint pair of functors studied thoroughly in \cite{A}. Consider the maps
\[ \mathcal{N}\stackrel{\mu}{\leftarrow}\widetilde{\mathcal{N}}\stackrel{\pi}{\rightarrow}\mathscr{B}.\] We define the functors
 \[ \Psi: \DBc\rightarrow\DNc \hspace{.5cm}\textup{ and }\hspace{.5cm} \Phi:\DNc\rightarrow\DBc\]
by $\Psi = \mu_{*}\pi^{!}$ and $\Phi = (\pi_{*}\mu^!)[-d](-\frac{d}{2})$. Note that $\Psi\simeq\mu_{!}\pi^{*}[d](\frac{d}{2})$ since $\mu$ is proper and $\pi$ is smooth of relative dimension $\frac{d}{2}$. It is easy to see that $(\Psi,\Phi)$ forms an adjoint pair. We denote the isomorphism induced by adjunction by
\[ \theta:\Hom^i(\constant_\mathscr{B}, \Phi(\mathbf{A}))\stackrel{\sim}{\rightarrow}\Hom^i(\Psi(\constant_\mathscr{B}), \mathbf{A}) = \Hom^i(\mathbf{A}, \mathbf{A}).\] %According to 

The following is a refinement \cite[Theorem 4.6]{A}:
\begin{lemma}\label{Lemma:purenfrinv} The category $\puren$ is Frobenius invariant. That is, for all $\IC_\chi, \IC_\psi$ in $\mathscr{S}$ and $i\in\mathbb{Z}$ \[\uHom(\IC_\chi, \IC_\psi[i](\textstyle\frac{i}{2}))^{\textup{Frob}}\cong \uHom(\IC_\chi, \IC_\psi[i](\textstyle\frac{i}{2})).\] In particular, $\uHom(\IC_\chi, \IC_\psi[i](\frac{i}{2})) \cong \Hom_{W}(V_\chi^*, \cH^i(\mathscr{B})\otimes V_\psi^*)$ and $\uHom^i(\IC_\chi, \IC_\psi)$ vanishes for $i$ odd.  
\end{lemma}

\begin{proof} Note that $\uHom^{i}(\mathbf{A}, \IC_\psi(\frac{i}{2})) \cong  \displaystyle\bigoplus_\chi \uHom^i(\IC_\chi, \IC_\psi(\textstyle\frac{i}{2}))\otimes V_\chi^*$. Thus, we may compute $\uHom^i(\IC_\chi, \IC_\psi(\frac{i}{2}))$ by the $\chi^*$-isotypic component of $\uHom^{i}(\mathbf{A}, \IC_\psi(\frac{i}{2}))$ under its $W$ action. Therefore, we have that
\begin{align*} 
\uHom^i(\IC_\chi, \IC_\psi(\textstyle\frac{i}{2})) &\cong \Hom_W (V_\chi^*, \uHom^{i}(\mathbf{A}, \IC_\psi (\textstyle\frac{i}{2})))\\
				& \cong \Hom_W (V_\chi^*, \cH^i(\mathscr{B})(\textstyle\frac{i}{2})\otimes\uHom(\constant_\mathscr{B}, \Phi (\IC_\psi)))\\
				& \cong \Hom_W (V_\chi^*, \cH^i(\mathscr{B})(\textstyle\frac{i}{2})\otimes\uHom(\mathbf{A}, \IC_\psi))\\ 
				& \cong \Hom_W (V_\chi^*, \cH^i(\mathscr{B})(\textstyle\frac{i}{2})\otimes\uHom(\IC_\psi\otimes V_\psi, \IC_\psi))\\ 
				& \cong \Hom_W (V_\chi^*, \cH^i(\mathscr{B})(\textstyle\frac{i}{2})\otimes V_\psi^*)\\
				& \cong (V_\chi\otimes\cH^i(\mathscr{B})(\textstyle\frac{i}{2})\otimes V_\psi^*)^W
\end{align*}

Recall that Frobenius acts on $\cH^i(\mathscr{B})$ by $q^{\frac{i}{2}}$. Hence, the Frobenius action on $\cH^i(\mathscr{B})(\frac{i}{2})$ is trivial. Since $\cH^i(\mathscr{B})$ vanishes for $i$ odd, so does $\uHom^i(\IC_\chi, \IC_\psi)$. \end{proof}

\begin{theorem}\label{theorem:Gequivmixed} We have that $\KN$ is a mixed version of $\DNbar$, where $\DNbar$ is the triangulated category generated by the image of $\spr$ in $\DNc$. Similarly, in the $G$ equivariant case, we have that $\KNG$ is a mixed version of $\DNGbar$, where $\DNGbar$ is the triangulated category generated by the image of $\spr$ in $\DNGc$. 
\end{theorem}

\begin{proof} By Lemma \ref{Lemma:purenfrinv}, we see that $\puren$ has Frobenius invariant morphisms. Now, we apply Theorem \ref{theorem:genmixedversion}. \end{proof}

\begin{corollary} The category $\Perv^{mix}(\mathcal{N})$ is a mixed version of the category of perverse sheaves $\spr\subset\DNbar$. Similarly, $\Perv^{mix}_G(\mathcal{N})$ is a mixed version of the category of perverse sheaves $\spr\subset\DNGbar$.\end{corollary}

\section{Mixed Springer Correspondence}\label{sec:mixedspringer} We now prove a \textit{mixed version} of the Springer correspondence, i.e. an equivalence of two mixed triangulated categories: $\KNG$, related to the geometry of $\mathcal{N}$, and $\dRG$, related to the representation theory of $W$. (We also consider the obvious non-equivariant analogue.)
  
Since we have a $W$-action on $\coflag$, we may define the smash product algebra $\R$. As a vector space, $\R = \Ql[W]\otimes\coflag$. The multiplication is given by 
\[ (w_1\otimes f_1)(w_2\otimes f_2) = w_1w_2\otimes\kappa(w_2^{-1})(f_1)f_2\]
for $w_1, w_2\in W$ and $f_1, f_2\in\coflag$. This algebra is discussed in \cite{DR} where they show that $\R$ and $H^{BM}_{2d-\bullet}(Z)$ are isomorphic as graded algebras. Of course, this combined with the isomorphism from \eqref{BMisom} implies that 
\[\Hom^\bullet(\mathbf{A}, \mathbf{A}) \cong \R\ .\] 

Now we consider the following $G$-equivariant version of the above proposition. A proof can be found in \cite{K}.

\begin{proposition}\label{prop:ringiso} The rings $\RG$ and $\Hom_G^\bullet(\mathbf{A},\mathbf{A})$ are isomorphic.
\end{proposition}

The above isomorphisms are isomorphisms of graded rings. In particular, the vanishing of the cohomology of the flag variety in odd degrees implies $\Hom^i(\mathbf{A}, \mathbf{A})$ and $\Hom_G^i(\mathbf{A}, \mathbf{A})$ vanish when $i$ is odd. Thus, we may consider the graded algebra $\mathcal{A}$ defined by $\mathcal{A}^i := \Hom^{2i}(\mathbf{A}, \mathbf{A})$. Similarly, we define $\mathcal{A}_G$ in the $G$-equivariant case. Then, it is easy to see that \[ \mathcal{A}\cong \Rtwo \textup{ and } \mathcal{A}_G\cong \RGtwo. \]

\subsection*{Graded $\mathcal{A}$ and $\mathcal{A}_G$ modules} Let $R$ be a positively graded ring, $R = \bigoplus_{i\in\mathbb{N}}R^i$. Let $\gmod(R)$ denote the category of finitely generated right graded modules over $R$. For a graded module $M\in\gmod(R)$ with $M = \bigoplus_{i\in\mathbb{Z}}M^i$, we denote by $\{1\}M$ the object obtained by shifting the grading $(\{1\}M)^i = M^{i+1}$. Let $\gproj(R)$ denote the full subcategory consisting of finitely generated projectives. Define $R^+ = \bigoplus_{i>0}R^i$, and suppose that the quotient $R_0 = R/R^+$ is a semisimple ring. \begin{remark}\label{projectives}In this case, the projective modules for $R$ are easy to describe. Let $L$ be a simple (right) $R_0$-module. Then $L\otimes_{R_0} R$ is a projective $R$-module. In fact, any projective $R$-module is a direct sum of shifts of these. A proof of this can be found in \cite[Lemma 6]{S1}. \end{remark}

In our case, we consider the graded rings $\mathcal{A} = \Rtwo$ and $\mathcal{A}_G=\RGtwo$. The degree 0 piece is isomorphic to $\QlW$ in both cases. Thus, for any irreducible $W$ representation $V_\chi$, a shift of $V_\chi\otimes\Coinv$ (or $V_\chi\otimes\Sh$ in the $G$ equivariant case) is an indecomposable projective module.  

% Consider the ideal $(\R)_{>0}$ generated by elements of strictly positive degree. Note that the quotient of the ring $\R$ by $(\R)_{>0}$ is semisimple. Indeed, $\R/(\R)_{>0}\cong \QlW$. Thus, listing the projective objects for this ring becomes  quite easy. If $V_\chi$ is an irreducible $W$ representation, then $V_\chi\otimes\coflag$ is an indecomposable projective. Direct sums of shifts of these make up all projectives for $\R$. A proof of this can be found in \cite{S1}.

Let $\dA$ and $\dAG$ be the bounded derived category of finitely generated graded modules over $\Rtwo$ and $\RGtwo$. We also consider the perfect derived category: $\dperA$. It is the full triangulated subcategory of $\dA$ generated by modules with a finite projective resolution. In other words, $\dperA$ is equivalent to $\hotproj$, the bounded homotopy category of (finitely generated) graded $\Rtwo$ modules that are projective. Note that $\dAG\cong\dperAG$ since $\RGtwo$ has finite homological dimension.  

\begin{theorem}[Mixed Springer Correspondence]\label{mixedSpringer} We have the following equivalences of triangulated categories: 
\begin{itemize}
	\item $\KN$ is equivalent to $\dperA$;
	\item $\KNG$ is equivalent to $\dAG$.
\end{itemize}
\end{theorem}

\begin{proof} Since we have equivalences $\dperA\cong\hotproj$ and $\dAG\cong\hotprojG$, it suffices to show that the additive categories $\Pure(\mathcal{N}_0)$ and $\gproj(\Rtwo)$ (or $\puregn$ and $\gproj(\RGtwo)$ in the $G$-equivariant case) are equivalent which is the content of the following proposition. \end{proof}
\begin{remark} The analogue of the above in the $G\times\mathbb{G}_\textup{m}$-equivariant setting should also hold by the same methods. In this case, the category $\KNGG$ should be equivalent to $\dAGG$, the derived category of graded modules over the graded Hecke algebra considered by Lusztig in \cite{L4}. We also note the discrepancy between the above statement and that in the introduction. The statement in the introduction holds with a slight modification of the definition of $\Pure(\mathcal{N}_0)$ and $\Pure_G(\mathcal{N}_0)$ to allow \textit{integral} shift-twist $[1](\frac{1}{2})$.
\end{remark}

\begin{proposition} \label{prop:proj}The categories $\Pure(\mathcal{N}_0)$ and $\gproj(\Rtwo)$ are equivalent as additive categories.
\end{proposition}

\begin{proof} We apply the functor $\varphi:=\bigoplus_{m\in\mathbb{Z}}\Hom_{\DN}(\mathbf{A}[-2m](-m), -)$. For an indecomposable $\IC_\chi[2i](i)$, we get \[\uHom^{2m+2i}(\mathbf{A}, \IC_\chi(i+m))\cong V_\chi^*\otimes\cH^{2m +2i}(\mathscr{B})(m+i)\] in degree $m$. Summing gives $\varphi(\IC_\chi[2i](i)) = \{i\}V_\chi^*\otimes\Coinv.$ It is easy to see that we get all objects in $\gproj(\Rtwo)$ in this way based on the Remark \ref{projectives} above. 

Now, we need to show that $\Hom(\IC_\chi, \IC_\psi[2i](i))\cong\Hom(\varphi(\IC_\chi), \varphi(\IC_\psi[2i](i)))$. Recall from the proof of Lemma \ref{Lemma:purenfrinv}, we have that \[\Hom(\IC_\chi, \IC_\psi[2i](i))\cong \Hom_W (V_\chi^*, \cH^{2i}(\mathscr{B})(i)\otimes V_\psi^*).\] It is easy to see that a $W$-equivariant map $V_\chi^*\rightarrow\cH^{2i}(\mathscr{B})(i)\otimes V_\psi^*$ uniquely determines a map of graded $\Rtwo$ modules from $V_\chi^*\otimes\Coinv\rightarrow \{i\}V_\psi^*\otimes\Coinv.$ \end{proof}

\begin{proposition} The categories $\Pure_G(\mathcal{N}_0)$ and $\gproj(\RGtwo)$ are equivalent as additive categories.
\end{proposition}

\begin{proof} Same as in Proposition \ref{prop:proj}.\end{proof} 
Note that the equivalence in Theorem \ref{mixedSpringer} does not preserve the usual $t$-stuctures, i.e. $\varphi(\Perv_{G}^{mix}(\mathcal{N}_0))\not\subset\gmod(\RGtwo)$. 

\section{Koszulity and DG modules}\label{dg}

Recall that the $G$-equivariant version of our Springer correspondence involves modules over the graded ring $\RGtwo$. Thus, it is natural to consider the additional structure of the Koszulity of this ring. Koszul duality between the symmetric and exterior algebras was first described in \cite{BGG}. This was extended to a more general class of rings in \cite{BGS} where they describe a derived equivalence between the categories of graded (finitely generated right) modules over a Koszul ring and its Koszul dual. Thus in our setting, we have an equivalence  
\[ \dRK \cong \dRG. \]

This equivalence transports the standard $t$-structure on $\dRK$ to a non-standard $t$-structure on $\dRG$. A description of this is given in \cite[2.13]{BGS}. Also recall the non-standard $t$-structure on $\KNG$ defined in Section \ref{section:tstr} with heart $\pervkn$. We will show that our mixed Springer equivalence \ref{mixedSpringer} is exact with respect to this non-standard $t$-structure. Therefore, it restricts to an equivalence of the hearts $\gmod(\RK)$ and $\pervkn$.

It is well known that the category $\gmod(\RK)$ has enough injective and projective objects. In particular, we see that $\pervkn$ also has enough injective and projective objects. Using this, we prove Theorem \ref{bigtheorem}: we have an equivalence of triangulated categories \[\DNGbar\cong\dgf(\RG),\] where $\dgf(\RG)$ is the derived category of finitely generated dg-modules over $\RG$.

The outline for the approach mostly follows that of \cite{S2}. 

\begin{enumerate}
	\item Take a projective resolution $\tilde{P}^\bullet\rightarrow\mathbf{A}$ in $\pervkn$.
	\item Let $P^\bullet$ be the image of $\tilde{P}^\bullet$ in $\DNGc$.
	\item Define the dgg-algebra $\dgHom(P^\bullet, P^\bullet)$.
	\item Show that the dg-ring $\dgHom(P^\bullet, P^\bullet)$ is formal and that its cohomology is $\RG$.
\end{enumerate}

\subsection{Non-standard $t$-structure}
We begin by showing our equivalence is exact with respect to the non-standard $t$-structure defined in \cite{BGS}.
\begin{theorem}\label{theorem:geomtstr} The non-standard $t$-structure defined in Section \ref{section:tstr} coincides with the geometric $t$-structure defined in \cite{BGS} on $\dRG$, thus the equivalence \[\KNG \cong \dRG\] is exact with respect to that $t$-structure. In particular, it restricts to an equivalence of the hearts \[\gmod(\RK)\cong\pervkn.\]
\end{theorem}
\begin{proof} As before, we let $\mathcal{A}_G = \RGtwo$. First, we recall the geometric $t$-structure, denoted $(D^{\leq0,g}, D^{\geq0,g})$, on $\dAG$ defined in \cite[Section 2.13]{BGS} obtained from the standard $t$-structure on $\dRK$ under the Koszul duality equivalence. The subcategory $D^{\leq0,g}\subseteq\dAG$ (respectively $D^{\geq0,g}\subseteq\dAG$) consists of objects isomorphic to complexes of graded projective modules \[\ldots\rightarrow P^i\rightarrow P^{i+1}\rightarrow\ldots\] such that $P^{i}$ is generated by its components of degree $\leq -i$ (respectively $\geq -i$) for all $i$. To show that our functor is exact with respect to this $t$-structure, it suffices to show that $\varphi(\pervkn)$ is contained within the heart $D^{\leq0,g}\cap D^{\geq0,g}$. Recall that an irreducible in $\pervkn$ has the form $(\IC_\chi[2i](i))[-i],$ i.e. a chain complex with $\IC_\chi[2i](i)$ in degree $i$ and $0$ elsewhere. Thus, the chain complex $\varphi((\IC_\chi[2i](i))[-i])$ is concentrated in degree $i$ with $\varphi((\IC_\chi[2i](i))[-i])^{i} = \{-i\}V_\chi\otimes\Sh$. It is easy to see that this is an object in the heart $D^{\leq0,g}\cap D^{\geq0,g}$.  \end{proof}

The equivalence $\dRK\cong\dRG$ proves that we have an equivalence \[\dpervk \cong \KNG.\] 

\subsection{A projective resolution and a dg-ring} Recall $\pervkn$ is the heart of the non-standard $t$-structure on $\KNG$ discussed in Section \ref{section:tstr}. By Theorem \ref{theorem:geomtstr}, $\pervkn$ has enough projectives. Let $\tilde{P}^\bullet$ be a projecive resolution of $\mathbf{A}$ \[(\cdots\rightarrow\tilde{P}^{-2}\rightarrow\tilde{P}^{-1}\rightarrow\tilde{P}^{0})\simeq \mathbf{A}\]
so each $\tilde{P}^i\in\pervkn\subseteq\KNG$. By Theorem \ref{theorem:Gequivmixed}, $\KNG$ is a mixed version of $\DNGbar$. In particular, we have a triangulated functor \[\nu : \KNG\rightarrow\DNGbar\] such that \[\displaystyle\bigoplus_{n\in\mathbb{Z}}\Hom_{\KNG}(M,N\langle2n\rangle) \xrightarrow{\sim}\Hom_{\DNGbar}(\nu M, \nu N)\]
for all $M, N\in\KNG$.

Now define the chain complex $P^\bullet\in\textup{C}^-\DNGbar$ by the following: $P^i=\nu(\tilde{P}^i)$ with differential $d_P$ given by the image of the differential of $\tilde{P}^\bullet$. 

Let $\mathcal{R}$ be the differential graded (dg) ring given by \[\mathcal{R}^n = \displaystyle\prod_{n=i+j, k\in\mathbb{Z}}\Hom_{\DNGbar}(P^{-i+k}, P^j[k])\] in degree $n$ and differential $d_\mathcal{R} f = d_P f - (-1)^{n}f d_P$ for $f$ homogeneous of degree $n$. We will refer to this grading as the vertical grading. For background on dg-rings and modules, see \cite{BL, S2}. Note that $\mathcal{R}$ has an extra grading (called the internal or horizontal grading) arising from the mixed structure of $\KNG$. For $m\in\mathbb{Z}$, we have \[\mathcal{R}^{n,2m} = \prod_{n=i+j, k\in\mathbb{Z}}\Hom_{\KNG}(\tilde{P}^{-i+k}\langle2m\rangle, \tilde{P}^j[k])\] and $\mathcal{R}^{n, 2m+1}=0.$ Also, note that the differential $d_\mathcal{R}$ respects the internal grading, i.e. $d_\mathcal{R}$ is a degree $(1, 0)$ map. The cohomology $\cH^\bullet(\mathcal{R})$ is a bigraded ring. We will regard it as a dg-ring with trivial differential in the vertical direction. 
\begin{lemma}\label{lemma:abovediagonal}The dg-ring $\mathcal{R}$ vanishes below the diagonal. 
\end{lemma}
\begin{proof} First, we show that $k=m$ since the $\tilde{P}$'s are projective in $\pervkn$. \begin{align*}\mathcal{R}^{n,2m} &= \prod_{n=i+j, k\in\mathbb{Z}}\Hom_{\KNG}(\tilde{P}^{-i+k}\langle2m\rangle, \tilde{P}^j[k])\\
 &=\prod_{n=i+j, k\in\mathbb{Z}}\Hom_{\KNG}(\tilde{P}^{-i+k}\tw{m}, \tilde{P}^j[k-m])\\
 &=\prod_{n=i+j, k\in\mathbb{Z}}\Ext^{k-m}_{\pervkn}(\tilde{P}^{-i+k}\tw{m}, \tilde{P}^j)\\
 &=\prod_{i\in\mathbb{Z}}\Hom_{\pervkn}(\tilde{P}^{m-i}\tw{m}, \tilde{P}^{n-i})\\ \end{align*} Now, in the mixed category $\pervkn$, $\tilde{P}^{m-i}\tw{m}$ has weights less than or equal to $2m-i$ with head pure of weight $2m-i$ and $\tilde{P}^{n-i}$ has weights less than or equal to $n-i$. Any morphism is strictly compatible with the weight filtration. Thus, for $f:\tilde{P}^{m-i}\tw{m}\rightarrow\tilde{P}^{n-i}$, we have that the head of the image of $f$ is pure of weight $2m-i$. This vanishes when $2m>n$.  \end{proof}

For any $M\in\DNGbar$, we define the dg-module $\dgHom(P^\bullet, M)$ over $\mathcal{R}$. In degree $i$, we have \[\dgHom(P^\bullet, M)^i = \displaystyle\prod_{j\in\mathbb{Z}}\Hom_{\DNGbar}(P^{-i+j}, M[j])\] and has differential induced by that of $P^\bullet$: if $f\in\dgHom(P^\bullet,M)^i$, then $d_{\tilde{M}}f = (-1)^{i+1}d_P f$. 
\begin{remark}Note that in each degree $i\in\mathbb{Z}$, the module $\dgHom(P^\bullet, M)$ has only finitely many non-zero terms. One can show directly that \[\dgHom(P^\bullet, \IC_\chi)^i = \Hom_{\pervkn}(\tilde{P}^{-i/2}\tw{i/2}, \IC_\chi)\] for $i$ even and vanishes for $i$ odd using properties of the $t$-structure and mixedness of $\pervkn$. D\'evissage proves the general case.
\end{remark}

\begin{theorem}\label{theorem:formal} The differential graded ring $\mathcal{R}$ is formal. In other words, $\mathcal{R}$ is quasi-isomorphic to its cohomology. And the cohomology ring is \[\cH^\bullet(\mathcal{R}) = \Hom^\bullet_{\DNGbar}(\mathbf{A}, \mathbf{A}) \cong \RG.\]
\end{theorem} 
\begin{proof} As a consequence of an idea of Deligne, \cite[5.3.1, Corollary 5.3.7]{D}, purity of the cohomology $\cH^\bullet(\mathcal{R})$ with respect to the internal grading implies formality of $\mathcal{R}$. A proof of this can be found in \cite[Proposition 4]{S2}.  Purity means the cohomology in vertical degree $i$ should be concentrated in horizontal degree $i$. In other words, $\cH^i(\mathcal{R}) = \cH^i(\mathcal{R}^{\bullet, i})$. Since $\dpervk \cong \KNG$ and $\tilde{P}^\bullet$ is a projective resolution of $\mathbf{A}$, we have that 
%Ext^i ( A<2m>, A) = Ext^i (A {m}[m]. A)
%= H^i Hom( P^{*+m} {m}, P^* ).

%With the old one, the last says instead H^i Hom( P^* {m}[m], P^* ),
%and that's zero for the same reason as before.
\begin{align*}
 \cH^i(\mathcal{R}^{\bullet, 2m}) &=\cH^i(\dgHom(\tilde{P}^{\bullet+m}\langle 2m\rangle, \tilde{P}^\bullet[m])\\
 																	&=\cH^i(\dgHom(\tilde{P}^{\bullet+m}\tw{m}, \tilde{P}^\bullet))\\
 																	&=\Ext^{i}_{\dpervk}(\mathbf{A}\tw{m}[m], \mathbf{A})\\
 																	&=\Hom^i_{\KNG}(\mathbf{A}\langle 2m\rangle, \mathbf{A}) \end{align*} Then we have
\begin{align*}
\Hom^i_{\KNG}(\mathbf{A}\langle 2m\rangle, \mathbf{A}) &= \Hom^{i-2m}_{\KNG}(\mathbf{A}[-2m](-m), \mathbf{A})\\
																		%&= \Hom^{i-2m}_{\KNG}(\mathbf{A}\langle 2m\rangle[-2m], \mathbf{A})\\
																		%&= \Hom^{i-2m}_{\DNG}(\mathbf{A}\langle 2m\rangle[-2m], \mathbf{A})\\
\end{align*}
In $\KNG$, $\mathbf{A}[-2m](-m)$ and $\mathbf{A}$ are chain complexes concentrated in degree 0. Clearly, if this is nonzero, then $i = 2m$. When $i = 2m$ since our realization functor $\KNG\rightarrow\DNG$ restricts to inclusion on $\puregn$, we have that \[\Hom_{\KNG}(\mathbf{A}[-2m](-m), \mathbf{A}) = \Hom_{\DNG}(\mathbf{A}[-2m](-m), \mathbf{A}).\] By Proposition \ref{prop:ringiso}, \[\Hom_{\DNG}(\mathbf{A}[-2m](-m), \mathbf{A}) = \uHom_{G}^{2m}(\mathbf{A}, \mathbf{A})(m)\cong \QlW\otimes\cH^{2m}_G(\mathscr{B}).\]\end{proof}
\begin{remark}\label{remark:qiso}In fact, more is true. By Lemma \ref{lemma:abovediagonal} and Schn\"urer's argument \cite[Proposition 4]{S2}, we have a dg-ring homomorphism $\mathcal{R}\rightarrow\RG$ which is a quasi-isomorphism. \end{remark}

Let $\dgder(\mathcal{R})$ and $\dgder(\RG)$ be the derived category of differential graded modules over $\mathcal{R}$ and $\RG$. %We regard $\RG$ as a dg-ring with trivial differential.
Let $\dgper(\RG)$ be the perfect derived category, i.e. the smallest full triangulated category generated by the free $\RG$ module and closed under direct summands. We note that $\dgper(\RG)\cong \dgf(\RG).$

Since $\mathcal{R}$ is quasi-isomorphic to $\RG$, we have an equivalence between the derived categories of dg-modules over these dg-rings. Let $\tilde{\mathcal{L}}$ denote the composition
\[\DNGbar \stackrel{\dgHom(P^\bullet, -)}{\xrightarrow{\hspace*{2cm}}}\dgder(\mathcal{R}) \simeq \dgder(\RG).\] Because $P^\bullet$ is not an object of a triangulated category, the definition of our functor $\dgHom(P^\bullet, -)$ is somewhat non-standard. We provide the following lemma to prove that $\dgHom(P^\bullet, -)$ commutes with shift [1],  and we prove the functor is triangulated in the appendix.

\begin{lemma} The functor $\dgHom(P^\bullet, -):\DNGc\rightarrow\dgder(\mathcal{R})$ commutes with shift. 
\end{lemma}
\begin{proof}  Recall that $\dgHom(P^\bullet, M)[1]^i = \dgHom(P^\bullet, M)^{i+1}$ with differential $d_{\tilde{M}[1]} = -d_{\tilde{M}}.$ 
\begin{align*}
	\dgHom(P^\bullet, M[1])^i &=\displaystyle\bigoplus_{j\in\mathbb{Z}}\Hom_{\DNGbar}(P^{-i+j}, M[j+1])\\    
	& =\displaystyle\bigoplus_{j\in\mathbb{Z}}\Hom_{\DNGbar}(P^{-i-1+j}, M[j])\\
	& =\dgHom(P^\bullet, M)[1]^i	
\end{align*}

The differential of $\dgHom(P^\bullet, M[1])$ is given by $d_{\tilde{M[1]}}f = (-1)^{i+1}d_P f$ for $f\in \dgHom(P^\bullet, M[1])^i = \dgHom(P^\bullet, M)^{i+1}$. So, $d_{\tilde{M[1]}}f = (-1)^{i+1}d_P f = -1((-1)^{i}d_P f = -d_{\tilde{M}}.$\end{proof}

\begin{lemma}\label{lem:tria} The functor $\dgHom(P^\bullet, -)$ is triangulated. \end{lemma}
\begin{proof} Given in the appendix. \end{proof}

\begin{lemma} The dg-module $\dgHom(P^\bullet, \mathbf{A})$ is quasi-isomorphic to the free module $\dgHom(P^\bullet, P^\bullet)$.
\end{lemma}

\begin{proof}
Let $\tilde{P}^\bullet\stackrel{q}{\rightarrow}\mathbf{A}$ be the quasi-isomorphism in $\pervkn$. The image of $q$ in $\DNGbar$ induces a morphism of dg-modules \[\dgHom(P^\bullet, P^\bullet)\stackrel{q^*}{\rightarrow}\dgHom(P^\bullet, \mathbf{A}).\]  

The fact that $q^*$ induces an isomorphism in cohomology is almost by definition since $\dgHom(P^\bullet, P^\bullet) = \bigoplus_{m\in\mathbb{Z}}\dgHom_{\KNG}(\tilde{P}^\bullet, \tilde{P}^\bullet\langle2m\rangle)$ and $\dgHom(P^\bullet, \mathbf{A}) = \bigoplus_{m\in\mathbb{Z}}\dgHom_{\KNG}(\tilde{P}^\bullet, \mathbf{A}\langle2m\rangle)$ and cohomology commutes with the direct sum; we provide the following computation anyway.
\begin{align*} \cH^i(\dgHom(P^\bullet, \mathbf{A})) &= \cH^i(\bigoplus_{m\in\mathbb{Z}}\dgHom_{\KNG}(\tilde{P}^\bullet, \mathbf{A}\langle2m\rangle))\\ 											&= \bigoplus_{m\in\mathbb{Z}}\cH^i(\dgHom_{\KNG}(\tilde{P}^\bullet, \mathbf{A}\langle2m\rangle))\\											&=\bigoplus_{m\in\mathbb{Z}}\Ext^i_{\pervkn}(\mathbf{A}, \mathbf{A}\langle2m\rangle)\\											 &=\bigoplus_{m\in\mathbb{Z}}\Hom_{\KNG}(\mathbf{A}, \mathbf{A}[-2m](-m)[2m+i])												
\end{align*}

Of course, $\Hom_{\KNG}(\mathbf{A}, \mathbf{A}[-2m](-m)[2m+i])\neq 0$ implies that $2m+i=0$ because otherwise we would have chain complexes concentrated in different degrees. Thus, we see that 
\begin{align*} \cH^i(\dgHom(P^\bullet, \mathbf{A})) &= \Hom_{\KNG}(\mathbf{A}, \mathbf{A}[i](\textstyle\frac{i}{2}))\\																											&= \Hom_{\DNG}(\mathbf{A}, \mathbf{A}[i](\textstyle\frac{i}{2}))\\
																										&\cong \uHom_{\DNGbar}(\mathbf{A}, \mathbf{A}[i](\textstyle\frac{i}{2}))\\ 						 &\cong \QlW\#\cH^i(\mathscr{B}).									
\end{align*}\end{proof}

\begin{theorem}[A derived Springer correspondence]\label{bigtheorem} The category $\DNGbar$ is equivalent to $\dgf(\RG)$ as a triangulated category.
\end{theorem}

\begin{proof} Recall that $\DNGbar$ is the triangulated category generated by $\spr$ in $\DNGc$ and that $\dgf(\RG)$ is the triangulated category generated by the summands of the free module $\RG$ in $\dgder(\RG)$. Thus, by \cite[Lemma 6]{S2} a refinement of Beilinson's Lemma \cite[Lemma 1.4]{B}, it suffices to prove that \[\Hom(\IC_\chi, \IC_\psi[i]) \cong \Hom(\tilde{\mathcal{L}}\IC_\chi, \tilde{\mathcal{L}}\IC_\psi[i])\] for all irreducible perverse sheaves $\IC_\chi, \IC_\psi\in\spr$, any $i\in\mathbb{Z}$. Note that $\tilde{\mathcal{L}}\IC_\chi\cong V_\chi^*\otimes\cH^\bullet_G(\mathscr{B})$ and $\tilde{\mathcal{L}}\IC_\psi[i]\cong V_\psi^*\otimes\cH^{\bullet+i}_G(\mathscr{B}).$ A morphism of dg-modules in this case is simply a morphism of graded modules since the differentials vanish. In fact, it is easy to see that such a morphism is determined by a $W$-equivariant map $V_\chi^*\rightarrow V_\psi^*\otimes\cH^i_G(\mathscr{B})$. Hence, \[\Hom(\tilde{\mathcal{L}}\IC_\chi, \tilde{\mathcal{L}}\IC_\psi[i])\cong\Hom_{W}(V_\chi^*, V_\psi^*\otimes\cH^i_G(\mathscr{B})).\] 

A slight modification (remove Tate twists) of Lemma \ref{Lemma:purenfrinv} proves that we also have \[\Hom(\IC_\chi, \IC_\psi[i])\cong\Hom_{W}(V_\chi^*, V_\psi^*\otimes\cH^i_G(\mathscr{B})).\] \end{proof}

\begin{remark} The usual Springer correspondence can be recovered by composing with $\cH^0$.
\end{remark}

\appendix
\section{Proof of Lemma \ref{lem:tria}}

Proving our functor is triangulated is not at all straightforward. The problem of course lies at the heart of the problem with triangulated categories: non-functoriality of cones. Thus, our approach is very roundabout. It is similar to successfully killing a mosquito: the mosquito must not see you coming.

First, we prove the restriction functor $i_*: \dgder(\RG)\rightarrow\dgder(\coGflag)$ induced by the injection $i:\coGflag\rightarrow\RG$ reflects distinguished triangles. 

Then we prove some straightforward facts about the flag variety. It is well known by work of \cite{BL} that the $G$-equivariant derived category of the flag variety is equivalent to a category of dg-modules $\dgder(\coGflag)$. We treat this as a sort of enhancement to the category $\DBGc$ and then use our adjoint pair $(\Psi, \Phi)$ to say something about the nilpotent cone. 

\subsection{Averaging}
Let $W$ be a finite group. Suppose that $f: V_1\rightarrow V_2$ is a linear map between $W$-representations that is not necessarily $W$-equivariant. Then we can easily produce a $W$-equivariant map $f^a$ by averaging: \[f^a(x) = \frac{1}{|W|}\displaystyle\sum_{w\in W}w^{-1}f(wx).\] It is straightforward to check that if $f$ is $W$-equivariant, then $f=f^a$. In fact, $f$ matches $f^a$ where $f$ is \textit{locally} $W$-equivariant. 
\begin{lemma}\label{local}Let $m\in V_1$ have the property that $f(wm)=wf(m)$ for all $w\in W.$ Then, $f^a(m) = f(m)$. 
\end{lemma}

\begin{proof} This is an easy computation.\end{proof}

\begin{lemma}\label{commutes}Suppose we have a commutative diagram 
\begin{center}
\begin{tikzpicture}[description/.style={fill=white,inner sep=1.5pt}] %[scale=.2]
 %\tikzstyle{every node}=[font=\small]

\matrix (m) [matrix of math nodes, row sep=2.5em,
column sep=1em, text height=1ex, text depth=0.25ex, nodes in empty cells]
{ M_1 & M_2 & M_3 \\
  N_1 & N_2 & N_3\\ };
\path[->,font=\scriptsize]
(m-1-1) edge node[above]{$g$} (m-1-2)
edge node[left]{$r$} (m-2-1)
(m-1-2) edge node[above]{$h$} (m-1-3)
edge node[left]{$f$} (m-2-2)
(m-1-3) edge node[left] {$s$} (m-2-3)
(m-2-1) edge node[above]{$g'$} (m-2-2)
(m-2-2) edge node[above]{$h'$} (m-2-3);
\end{tikzpicture}
\end{center}
so that $M_i$, $N_i$ are $W$-representations, the maps $g, h, g', h', r, s$ are $W$-equivariant, and $f$ is linear (not necessarily $W$-equivariant). Then the diagram where we replace $f$ with its average $f^a$ commutes.
\end{lemma}

\begin{proof}
We begin with the left square. Note that $fg(wm) = g'r(wm) = wg'r(m) = wfg(m).$ Thus, $f$ is \textit{locally $W$-equivariant} on the image of $g$. By Lemma \ref{local}, $f^ag = fg = g'r.$

Now, we consider the right square. Note that the composition $h'f$ is $W$-equivariant since it equals the $W$-equivariant map $sh$. 
\begin{align*}
	h'f^a(m)=\frac{1}{|W|}\displaystyle\sum_{w\in W}h'(w^{-1}f(wm))=\frac{1}{|W|}\displaystyle\sum_{w\in W}w^{-1}h'f(wm) =h'f(m).
\end{align*} Since $h'f = sh$, the result follows. \end{proof}

For a dg-ring $\mathcal{R}$, we denote by $\dgK(\mathcal{R})$ the homotopy category of homotopically projective dg-modules (also referred to as $K$-projectives). It is well known that we have an equivalence $\dgder(\mathcal{R})\cong\dgK(\mathcal{R})$ \cite[Corollary 10.12.2.9]{BL}.
  
We regard $\coGflag$ as a dg-ring with trivial differential. There is a natural functor $i_*: \dgder(\RG)\rightarrow\dgder(\coGflag)$ forgetting the $W$ action. This functor, although clearly not an equivalence, has a very special property: it reflects triangles. The meaning of this is the following:

\begin{proposition}\label{reflects}Let $L\stackrel{f}{\rightarrow}M\stackrel{g}{\rightarrow}N\stackrel{h}{\rightarrow}L[1]$ be a sequence in $\dgder(\RG)$ $(\cong\dgK(\RG))$ such that its image under $i_*$ is a distinguished triangle in $\dgder(\coGflag)$ $(\cong\dgK(\coGflag)$. Then it is also distinguished in $\dgder(\RG)$.
\end{proposition}

\begin{proof} Let $L\stackrel{f}{\rightarrow}M\stackrel{g}{\rightarrow}N\stackrel{h}{\rightarrow}L[1]$ be a candidate triangle in $\dgK(\RG)$ that becomes distinguished in $\dgK(\coGflag)$. By the hypothesis, we have an isomorphism between our candidate triangle and a standard triangle in the category $\dgK(\coGflag)$
\begin{center}
\begin{tikzpicture}[description/.style={fill=white,inner sep=1.5pt}] %[scale=.2]
 %\tikzstyle{every node}=[font=\small]

\matrix (m) [matrix of math nodes, row sep=2.5em,
column sep=1em, text height=1ex, text depth=0.25ex, nodes in empty cells]
{ L & M & N & L[1] \\
  L & M & cone(f) & L[1]\\ };
\path[->,font=\scriptsize]
(m-1-1) edge node[above]{$f$} (m-1-2)
edge node[above, sloped]{$=$} (m-2-1)
(m-1-2) edge node[above]{$g$} (m-1-3)
edge node[above, sloped]{$=$} (m-2-2)
(m-1-3) edge node[above]{$h$} (m-1-4)
edge node[left] {$q$} (m-2-3)
(m-1-4) edge node[above, sloped] {$=$} (m-2-4)
(m-2-1) edge node[above]{$f$} (m-2-2)
(m-2-2) edge node[above]{$\pi_1$} (m-2-3)
(m-2-3) edge node[above]{$\pi_2$} (m-2-4);
\end{tikzpicture}
\end{center} where the map $q$ is not necessarily $W$-equivariant. Note that the standard triangle $L\stackrel{f}{\rightarrow}M\stackrel{\pi_1}{\rightarrow}cone(f)\stackrel{\pi_2}{\rightarrow}L[1]$ is distinguished in $\dgK(\RG)$.  %It is clear that the sequence  $L\stackrel{f}{\rightarrow}M\stackrel{g}{\rightarrow}N\stackrel{h}{\rightarrow}L[1]$ gives a long exact sequence in cohomology since it is distinguished in $\dgder(\coGflag)$ and since $\cH^i$ commutes with $\mathbb{U}$. 
To get a morphism of triangles in $\dgK(\RG)$, we replace $q$ with $q^a$ to get the following:

\begin{center}
\begin{tikzpicture}[description/.style={fill=white,inner sep=1.5pt}] %[scale=.2]
 %\tikzstyle{every node}=[font=\small]

\matrix (m) [matrix of math nodes, row sep=2.5em,
column sep=1em, text height=1ex, text depth=0.25ex, nodes in empty cells]
{ L & M & N & L[1] \\
  L & M & cone(f) & L[1]\\ };
\path[->,font=\scriptsize]
(m-1-1) edge node[above]{$f$} (m-1-2)
edge node[above, sloped]{$=$} (m-2-1)
(m-1-2) edge node[above]{$g$} (m-1-3)
edge node[above, sloped]{$=$} (m-2-2)
(m-1-3) edge node[above]{$h$} (m-1-4)
edge node[left] {$q^a$} (m-2-3)
(m-1-4) edge node[above, sloped] {$=$} (m-2-4)
(m-2-1) edge node[above]{$f$} (m-2-2)
(m-2-2) edge node[above]{$\pi_1$} (m-2-3)
(m-2-3) edge node[above]{$\pi_2$} (m-2-4);
\end{tikzpicture}
\end{center}

The fact that each square commutes is proven in Lemma \ref{commutes}. We take cohomology in $\dgK(\RG)$ to get a long exact sequence and then apply the five lemma. This proves that $\cH^i(q^a)$ is an isomorphism for all $i\in\mathbb{Z}$. Thus, $q^a$ is a quasi-isomorphism. \end{proof}

\subsection{Statements for the flag variety}
\begin{definition} Define $\Pure_G(\mathscr{B}_0)$ as the full subcategory of $\DBG$  with objects in $\displaystyle\bigoplus_{i\in\mathbb{Z}}\constant_{\mathscr{B}_0}[2i](i)$. 
\end{definition}

\begin{lemma} Let $\Pure_G(\mathscr{B}_0)$ be defined as above. Then we have
\begin{enumerate}
	\item $\KBG$ is a mixed version of $\DBGc$.
	\item $\KBG \cong \dSh$ as triangulated categories.
	\item $\KBG$ has a non-standard $t$-structure with heart $\pervk(\mathscr{B}_0)$ containing enough projectives.
	\item Choose a projective resolution $\tilde{Q}^\bullet$ of $\constant_{\mathscr{B}_0}$ in $\pervk(\mathscr{B}_0)$ such that the image $Q^\bullet$ in $\textup{C}^-\DBGc$ satisfies $\Psi(Q^\bullet) = P^\bullet$. Define $\dgHom(Q^\bullet, Q^\bullet)$ as above. Then $\dgHom(Q^\bullet, Q^\bullet)$ is formal.
\end{enumerate}
\end{lemma}

\begin{proof} The proof of each statement is exactly the same as the analogous statement above for $\mathcal{N}$.\end{proof}

Let $\gamma:\DBGc\stackrel{\sim}{\rightarrow}\dgf(\coGflag)$ denote the equivalence described in \cite[Theorem 2.6.3, Theorem 12.7.2 \textit{ii}]{BL}. Note that it is given by composition \[\DBGc\stackrel{\sim}{\rightarrow}D^b_{B,c}(pt)\cong\dgf(\coGflag)\] where the first functor is induction equivalence. Define the chain complex of dg-modules $\mathscr{Q}^\bullet$ by $\mathscr{Q}^i = \gamma(Q^i)$ and differential $d_{\mathscr{Q}}=\gamma(Q^i\rightarrow Q^{i+1})$. Now we define a functor\[\mathcal{L}_\mathscr{B}:\dgder(\coGflag)\rightarrow\dgder(\dgHom(Q^\bullet, Q^\bullet))\] by the following formula $\mathcal{L}_\mathscr{B}(M)^n = \displaystyle\bigoplus_{n=j-i}\Hom_{\dgder(\coGflag)}(\mathscr{Q}^i,M[j])$ with differential given by $d_{\mathcal{L}_\mathscr{B}(M)}f = d_M f-(-1)^{n}fd_{\mathscr{Q}}$ for $f$ homogeneous of degree $n$. Note that the dg-rings $\dgHom(\mathscr{Q}^\bullet, \mathscr{Q}^\bullet)$ and $\dgHom(Q^\bullet, Q^\bullet)$ are isomorphic as dg-rings (not just quasi-isomorphic).

\begin{lemma} The functor $\mathcal{L}_\mathscr{B}:\dgder(\coGflag)\rightarrow\dgder(\dgHom(Q^\bullet, Q^\bullet))$ is triangulated.% and the restriction to the perfect derived category $\mathcal{L}_\mathscr{B}:\dgper(\coGflag)\rightarrow\dgper(\dgHom(Q^\bullet, Q^\bullet))$ is an equivalence.
\end{lemma}

\begin{proof} To see that the functor is triangulated, we will prove that for a morphism $M\stackrel{f}{\rightarrow}N$ of $\coGflag$ dg-modules, $\mathcal{L}_\mathscr{B}(cone(f))\cong cone(\mathcal{L}_\mathscr{B}(f))$. Let \[g =\left(\begin{array}{c} 
g_1\\ g_2\end{array}\right)\in cone(\mathcal{L}_\mathscr{B}(f))\] be homogeneous of degree $i$. Note that this implies $g_1$ is homogeneous of degree $i+1$ for $\mathcal{L}_{\mathscr{B}}(M)$. Of course, $cone(\mathcal{L}_\mathscr{B}(f))=\mathcal{L}_\mathscr{B}(M)[1]\oplus\mathcal{L}_\mathscr{B}(N)$ with differential given by  
\[d_1(g) =\left( \begin{array}{cc} 
-d_{\mathcal{L}(M)} & 0\\
\mathcal{L}(f) & d_{\mathcal{L}(N)}\end{array} \right)\left(\begin{array}{c} 
g_1\\ g_2\end{array}\right) =  \left( \begin{array}{cc} 
-d_M g_1 - (-1)^i g_1d_Q\\
fg_1 + d_Ng_2 - (-1)^ig_2d_Q\end{array} \right)\] On the other hand, the induced differential applied to $g\in\mathcal{L}_{\mathscr{B}}(cone(f))$ is given by \[d_2(g) =\left( \begin{array}{cc} 
-d_M & 0\\
f & d_N\end{array}\right)\left(\begin{array}{c} 
g_1\\ g_2\end{array}\right) - (-1)^i\left( \begin{array}{c} 
g_1\\
g_2\end{array}\right)d_Q\] It is clear that these two are the same. \end{proof}

\begin{remark} There is an isomorphism of functors between \[\mathcal{L}_\mathscr{B}\circ\gamma :\DBGc \rightarrow \dgder(\dgEnd(Q^\bullet)) \textup{ and } \dgHom(Q^\bullet, -): \DBGc \rightarrow \dgder(\dgEnd(Q^\bullet)).\] To see this, it suffices to compare the differentials of the dg-modules $\mathcal{L}_\mathscr{B}\circ\gamma(M)$ and $\dgHom(Q^\bullet, M)$ for $M\in\DBGc$. Note that $\dgHom(Q^\bullet, -)$ is defined as before: for $M\in\DBGc$, we define the dg-module $\dgHom(Q^\bullet, M)$ over $\dgEnd(Q^\bullet)$. In degree $i$, we have \[\dgHom(Q^\bullet, M)^i = \displaystyle\bigoplus_{j\in\mathbb{Z}}\Hom_{\DBGc}(Q^{-i+j}, M[j])\] and has differential induced by that of $Q^\bullet$: if $f\in\dgHom(Q^\bullet,M)^i$, then $d_{\tilde{M}}f = (-1)^{i+1}d_Q f$. The differential for $\mathcal{L}_\mathscr{B}(M)$ is given by  $d_{\mathcal{L}_\mathscr{B}(M)}f = d_{\gamma(M)} f-(-1)^{i}fd_{\mathscr{Q}}.$ These two differ by $d_{\gamma (M)} f$, but $d_{\gamma (M)}$ is a null-homotopic map of dg-modules. Since  $\mathcal{L}_\mathscr{B}\circ\gamma(M)$ is a direct product of $\Hom$ groups in $\dgder(\coGflag),$ composition with $d_{\gamma (M)}$ takes maps to zero. 
\end{remark}
Recall the adjoint pair $(\Psi, \Phi)$ defined in Section \ref{adjointpair}. By construction, we have an isomorphism of chain complexes $P^\bullet = \Psi(Q^\bullet)$. Hence $\Psi$ induces a map of dg-rings $\Psi: \dgEnd(Q^\bullet)\rightarrow \dgEnd(P^\bullet)$. Similarly we have homomorphisms of dg-rings $\pi_1:\dgEnd(P^\bullet)\rightarrow \RG$, $i:\coGflag\rightarrow\RG$, and $\pi_2:\dgEnd(Q^\bullet)\rightarrow\coGflag$. (We note the existence of $\pi_1$ and $\pi_2$ is implied by Remark \ref{remark:qiso}.) These are related by the following lemma.

\begin{lemma}\label{commutes} Let $i$, $\pi_1$,  $\pi_2$, and $\Psi$ be defined as above. Then $i\circ\pi_2 = \pi_1\circ\Psi.$
\end{lemma}

\begin{proof}
We need to show the following diagram commutes.
\begin{center}
\begin{tikzpicture}
\small
\node (F) at (4, 0) {$\dgEnd(Q^\bullet)$};
\node (G) at (8, 0) {$\Hom^\bullet(\constant_\mathscr{B}, \constant_\mathscr{B})$};
\node (B) at (4, -2) {$\dgEnd(P^\bullet)$}; 
\node (C) at (8, -2) {$\Hom^\bullet(\mathbf{A}, \mathbf{A})$}; 
\draw[->] (F) -- (G)
node[midway, above] {$\pi_2$};
\draw[->] (F) -- (B) 
node[midway, left] {$\Psi$};
\draw[->] (B) -- (C)
node[midway, above] {$\pi_1$}; 
\draw[->] (G) -- (C)
node[midway, left] {$i$}; 
\end{tikzpicture}
\end{center}
Note that the inclusion $i$ is induced by the functor $\Psi$. We may rewrite the above diagram in the following way:
\begin{center}
\begin{tikzpicture}
\small
\node (F) at (4, 0) {$\dgEnd(Q^\bullet)$};
\node (G) at (8, 0) {$\cH^\bullet(\dgEnd(Q^\bullet))$};
\node (B) at (4, -2) {$\dgEnd(\Psi(Q)^\bullet)$}; 
\node (C) at (8, -2) {$\cH^\bullet(\dgEnd(\Psi(Q)^\bullet)))$}; 
\draw[->] (F) -- (G)
node[midway, above] {$\pi_2$};
\draw[->] (F) -- (B) 
node[midway, left] {$\Psi$};
\draw[->] (B) -- (C)
node[midway, above] {$\pi_1$}; 
\draw[->] (G) -- (C)
node[midway, left] {$i$}; 
\end{tikzpicture}
\end{center}

Note that $\Psi$ is a chain map $\dgEnd(Q^\bullet)\rightarrow\dgEnd(P^\bullet)$. The differential on $\dgEnd(Q^\bullet)$ is given by $d_Qf-(-1)^nfd_Q$. After we apply $\Psi$, we get $\Psi(d_Q)\Psi(f)-(-1)^n\Psi(f)\Psi(d_Q)$. Since $\Psi(d_Q)=d_P$, we have that \[\Psi(d_Q)\Psi(f)-(-1)^n\Psi(f)\Psi(d_Q) = d_P\Psi(f)-(-1)^n\Psi(f)d_P,\] which is exactly what we get when we apply the differential of $\dgEnd(P^\bullet)$ to $\Psi(f)$.

Let $f\in\dgEnd(Q^\bullet)^n$. Let $\bar{f}$ denote the cohomology class determined by $f$, i.e. $i\circ\pi_2(f) = \Psi(\bar{f})$. We want to compare this to $\overline{\Psi(f)}$.
\begin{align*}\overline{\Psi(f)} - \Psi(\bar{f}) &= \Psi(f) + Im(d_P) - \Psi(f + Im(d_Q))\\
																						&= \Psi(f) + Im(d_P) - \Psi(f) - \Psi(Im(d_Q))\\
																					 &= Im(d_P)\end{align*}

since $\Psi(Im(d_Q))\subset Im(d_P)$ since $\Psi d_Q = d_P \Psi$. Of course, $Im(d_P) = 0$ in $\cH^\bullet(\dgEnd(P^\bullet)).$ \end{proof}

\begin{lemma} The functor $\tilde{\mathcal{L}}:\DNGbar\rightarrow\dgder(\RG)$ (and hence $\mathcal{L}:\DNGbar\rightarrow\dgder(\dgEnd(P^\bullet))$) is triangulated.
\end{lemma}
\begin{proof}

We denote by $\Psi_*$, $\pi_{1 *}$, $\pi_{2 *}$, and $i_*$ the restriction functors defined in \cite[10.12.5]{BL} induced by the corresponding dg-ring homomorphisms. Lemma \ref{commutes} implies that the following diagram commutes:
\begin{center}
\begin{tikzpicture}
\small
\node (A) at (0, 0) {$\DNGbar$};
\node (B) at (4, 0) {$\dgder(\dgEnd(P^\bullet))$}; 
\node (C) at (8, 0) {$\dgder(\RG)$}; 
\node (D) at (0,-2) {$\DBGc$};
%\node (E) at (3, -2) {$\dgder(\coGflag)$};
\node (F) at (4, -2) {$\dgder(\dgEnd(Q^\bullet))$};
\node (G) at (8, -2) {$\dgder(\coGflag)$};
\draw[->] (A) -- (B)
node[midway, above] {$\mathcal{L}$};
\draw[->] (C) -- (B) 
node[midway, above] {$\pi_{1 *}$};
\draw[->] (B) -- (F)
node[midway, left] {$\Psi_*$}; 
\draw[->] (C) -- (G)
node[midway, left] {$i_*$}; 
\draw[->] (A) -- (D)
node[midway, left] {$\Phi$};
\draw[->] (D) -- (F)
node[midway, above] {$\mathcal{L}_\mathscr{B}$};
\draw[->] (G) -- (F)
node[midway, above] {$\pi_{2 *}$};
\end{tikzpicture}
\end{center}

Now, since $\pi_{1 *}$ and $\pi_{2 *}$ are equivalences, the following diagram where we replace them with their inverses also commutes.
\begin{center}
\begin{tikzpicture}
\small
\node (A) at (0, 0) {$\DNGbar$};
\node (B) at (4, 0) {$\dgder(\dgEnd(P^\bullet))$}; 
\node (C) at (8, 0) {$\dgder(\RG)$}; 
\node (D) at (0,-2) {$\DBGc$};
%\node (E) at (3, -2) {$\dgder(\coGflag)$};
\node (F) at (4, -2) {$\dgder(\dgEnd(Q^\bullet))$};
\node (G) at (8, -2) {$\dgder(\coGflag)$};
\draw[->] (A) -- (B)
node[midway, above] {$\mathcal{L}$};
\draw[->] (B) -- (C) 
node[midway, above] {$\pi^*_1$};
\draw[->] (B) -- (F)
node[midway, left] {$\Psi_*$}; 
\draw[->] (C) -- (G)
node[midway, left] {$i_*$}; 
\draw[->] (A) -- (D)
node[midway, left] {$\Phi$};
\draw[->] (D) -- (F)
node[midway, above] {$\mathcal{L}_\mathscr{B}$};
\draw[->] (F) -- (G)
node[midway, above] {$\pi_2^*$};
\end{tikzpicture}
\end{center}

Thus, we have that the outer diagram commutes, where $\tilde{\mathcal{L}} = \pi^*_1\circ\mathcal{L}$ and $\tilde{\mathcal{L}}_{\mathscr{B}} = \pi_2^*\circ\mathcal{L}_\mathscr{B}$. 
\begin{center}
\begin{tikzpicture}
\small
\node (A) at (0, 0) {$\DNGbar$};
\node (C) at (4, 0) {$\dgder(\RG)$}; 
\node (D) at (0,-2) {$\DBGc$};
\node (G) at (4, -2) {$\dgder(\coGflag)$};
\draw[->] (A) -- (C)
node[midway, above] {$\tilde{\mathcal{L}}$};
\draw[->] (C) -- (G)
node[midway, right] {$i_*$}; 
\draw[->] (A) -- (D)
node[midway, right] {$\Phi$};
\draw[->] (D)	-- (G)
node[midway, above] {$\tilde{\mathcal{L}}_{\mathscr{B}}$};
\end{tikzpicture}
\end{center} Suppose $L\rightarrow M\rightarrow N\rightarrow L[1]$ is distinguished in $\DNGbar$. We need to show that the triangle $\tilde{\mathcal{L}}L\rightarrow\tilde{\mathcal{L}}M\rightarrow\tilde{\mathcal{L}}N\rightarrow\tilde{\mathcal{L}}L[1]$ is distinguished. Note that the triangle $\tilde{\mathcal{L}}_{\mathscr{B}}\Phi L\rightarrow\tilde{\mathcal{L}}_{\mathscr{B}}\Phi M\rightarrow\tilde{\mathcal{L}}_{\mathscr{B}}\Phi N\rightarrow\tilde{\mathcal{L}}_{\mathscr{B}}\Phi L[1]$ is distinguished since $\tilde{\mathcal{L}}_{\mathscr{B}}\Phi$ is a composition of triangulated functors. By commutativity of the above square, this implies the triangle $i_*\tilde{\mathcal{L}}L\rightarrow i_*\tilde{\mathcal{L}}M\rightarrow i_*\tilde{\mathcal{L}}N\rightarrow i_*\tilde{\mathcal{L}}L[1]$ is distinguished. By Lemma \ref{reflects}, the functor $i_*$ reflects triangles. Hence, $\tilde{\mathcal{L}}L\rightarrow\tilde{\mathcal{L}}M\rightarrow\tilde{\mathcal{L}}N\rightarrow\tilde{\mathcal{L}}L[1]$ is distinguished. \end{proof}

\bibliographystyle{jabrefpna}	% (uses file "plain.bst") amsalpha alphanum
\bibliography{somereferences2}		% expects file "myrefs.bib"

\end{document}